\documentclass[11pt,a4paper]{article}
\usepackage{fullpage}

\usepackage[bitstream-charter]{mathdesign}

\DeclareSymbolFont{usualmathcal}{OMS}{cmsy}{m}{n}
\DeclareSymbolFontAlphabet{\mathcal}{usualmathcal}

\usepackage{xcolor}
\usepackage{hyperref}
\usepackage{graphicx}
\usepackage{physics}
\usepackage{amsthm}
\usepackage{algorithm}
\usepackage{algorithmicx}
\usepackage{algpseudocode}
\usepackage{enumitem}

\usepackage{mathabx}
\usepackage{enumitem}
\usepackage{subcaption}
\newcommand{\R}{\mathbb{R}}
\newcommand{\lift}{\mathrm{lift}}
\newcommand{\St}{\mathrm{St}}
\newcommand{\Gr}{\mathrm{Gr}}
\renewcommand{\grad}{\mathrm{grad}\,}
\newcommand{\cT}{\mathcal{T}}
\newcommand{\cZ}{\mathcal{Z}}
\newcommand{\cTG}{\mathcal{T/G}}
\newcommand{\Proj}{\mathrm{Proj}}
\newcommand{\DD}{\mathrm{D}}

\usepackage[top=12mm,bottom=12mm,left=30mm,right=30mm,head=12mm,includeheadfoot]{geometry}

\usepackage{tocloft}

\usepackage[
  backend=biber,
  eprint=true,
  doi=true,
  url=false,
  sorting=none
]{biblatex}
\AtEveryBibitem{\clearfield{issn}}
\theoremstyle{definition}
\newtheorem{theorem}{Theorem}

\theoremstyle{definition}
\newtheorem{proposition}[theorem]{Proposition}

\theoremstyle{definition}
\newtheorem{corollary}[theorem]{Corollary}

\newcommand{\Title}{Stochastic Optimization of Tree Tensor Networks}

\begin{document}

\begin{center}{\Large \textbf{
\Title
\\
}}\end{center}

\begin{center}\textbf{
Marius Willner\textsuperscript{12$\star$},
Maximilian Scharf\textsuperscript{23$\dagger$},
Andr\'e Uschmajew\textsuperscript{14},\\
Timo Felser\textsuperscript{2}
and Marco Trenti\textsuperscript{2}
}\end{center}

\begin{center}
{\bf 1} Institute of Mathematics, University of Augsburg, 86159 Augsburg, Germany
\\
{\bf 2} Tensor AI Solutions GmbH, 89284 Pfaffenhofen an der Roth, Germany
\\
{\bf 3} Institute for Complex Quantum Systems, Ulm University, 89081 Ulm, Germany
\\
{\bf 4}
Centre for Advanced Analytics and Predictive Sciences, University of Augsburg, \\86159 Augsburg, Germany
\\[\baselineskip]

$\star\dagger$ These authors contributed equally to this work.
\\
$\star$ \href{mailto:marius.willner@uni-a.de}{\small marius.willner@uni-a.de}\,,\quad
$\dagger$ \href{mailto:maximilian.scharf@tensor-solutions.com}{\small maximilian.scharf@tensor-solutions.com}
\end{center}

\section*{Abstract}
\textbf{\boldmath{%
Tensor networks, originally developed for quantum many-body physics, are promising models for machine learning. We derive stochastic Riemannian optimizers for tree tensor networks (TTNs) on both their parameter and quotient manifolds, including adaptive and learning-rate-free schemes suitable for minibatch training. Using a hybrid CNN–TTN architecture, we evaluate the methods on Fashion-MNIST, CIFAR10, and Imagenette. The proposed optimizers achieve predictive performance comparable to unconstrained optimization while enabling numerically stable downstream compression. 
}}

\vspace{10pt}
\hrule
\tableofcontents

\section{Introduction}
\label{sec:intro}

Modern machine learning relies overwhelmingly on first-order optimization, and in practice this means \emph{stochastic} gradient descent (SGD) rather than full-batch gradient descent. Given a loss that decomposes as an expectation over data, $f(\theta) = \frac{1}{m}\sum_{k=1}^{m}\ell(\theta;x^k, y^k)$, the exact gradient $\nabla f$ costs $O(m)$ per step, which is prohibitive for large $m$. SGD replaces it with an unbiased estimate computed on a minibatch $\mathcal{B}$,
\[
  \nabla f(\theta)\;\approx\;\frac{1}{|\mathcal{B}|}
  \sum_{k\in\mathcal{B}}\nabla_\theta \,\ell(\theta;x^k, y^k),
\]
trading gradient accuracy for a large gain in per-step cost~\cite{robbins1951,bottou2018}. Beyond efficiency, the injected gradient noise has genuine optimization value: it helps escape saddle points and sharp minima, and is widely associated with improved generalization in overparameterized models~\cite{keskar2017,jin2017}. Adaptive variants such as ADAM~\cite{kingma2015} or MUON~\cite{Keller2024} and distance-over-gradient schemes~\cite{Ivgi2023}
refine the basic method, and its convergence behavior for convex and nonconvex objectives is by now well characterized~\cite{ghadimi2013}. Together these properties have made stochastic first-order optimization the default training paradigm across essentially all of large-scale machine learning.

Tensor networks---matrix product states (MPS), tree tensor networks (TTN), and related structures developed for quantum many-body physics~\cite{white1992, FelserMontangero2026}---have enabled simulations of challenging higher-dimensional many-body and lattice-gauge systems \cite{Felser2020QED,Felser2021HighDim,Magnifico2021QED}, and have subsequently been repurposed as trainable models for supervised classification and generative modeling~\cite{stoudenmire2016,han2018}. Applications of TN-based learning already include high-energy physics data analysis \cite{felser2020quantuminspired, Borella_2025} and, more recently, robust and compressible classification of synthetic-aperture-radar data \cite{scharf2026quantuminspiredrobustscalablesar}.

The canonical optimizer inherited from physics is the density-matrix renormalization group (DMRG), a deterministic algorithm that sweeps across the network, locally optimizing one or two tensors at a time to near-optimality~\cite{white1992}. 
DMRG is extraordinarily effective for computing ground states, but its central assumption---a fixed objective whose local environment can be formed and factorized exactly at each step---sits uneasily with the machine-learning setting, where the objective is an empirical average over a large, shuffled dataset. Evaluating or differentiating that average exactly reintroduces the same $O(m)$ cost that motivates stochastic methods in the first place~\cite{robbins1951,bottou2018}, while the nonconvex, multilinear loss landscapes that tensor networks induce are precisely the regime in which gradient noise aids optimization and generalization~\cite{ghadimi2013,jin2017,keskar2017}. 

Stochastic gradient methods address both issues directly, and a growing body of work trains tensor-network models this way. The dominant approach obtains gradients by \emph{automatic differentiation}: the gradient of the loss with respect to each core tensor is assembled by backpropagation through the network contraction and passed to a generic stochastic optimizer, as in~\cite{efthymiou2019} and in recent training libraries built on deep-learning frameworks. This route is convenient and integrates tensor-network layers end-to-end with conventional neural components~\cite{stoudenmire2016,efthymiou2019,Hao2024}, but it treats the network as a generic parameterized function and ignores the geometric structure of the underlying decomposition. A second, smaller line of work instead respects that structure, performing Riemannian optimization on the fixed-rank manifold of tensor-train models~\cite{novikov2015} and TTNs~\cite{DaSilva2015,Hauru2021,Willner2025}; such methods, however, have largely been developed in the full-batch context. Recent advances in the closely related field of orthogonal optimization \cite{Mataigne2026, Javaloy2026, Li2026}, such as the POGO algorithm \cite{Javaloy2026}, serve as a basis to extend on those methods.

In this paper, we work out the forms of the stochastic algorithms on the manifolds of TTNs. Rather than indiscriminately using stochastic optimizers originally meant for Euclidean optimization,
we derive the Riemannian gradient of the loss in terms of the network's tensors and its canonical gauge structure, yielding explicit,
minibatch-computable update rules that respect the TTN (quotient) geometry. Crucially, we use that the TTN parameter space factors as a Cartesian product of manifolds, enabling adaptive Riemannian optimization. Even though the TTN quotient does not exhibit a product structure, we show that it shares its geodesic spray with a Cartesian product of Grassmann manifolds. We use this to theoretically back up our considerations and to conclude about the exponential map and Riemannian distances, which are central components of our stochastic optimizers. 
Furthermore, we devise a novel hybrid CNN-TTN architecture, that allows the numerical evaluation of the resulting algorithms on large-scale datasets, such as CIFAR10 and Imagenette:
they perform competitively with unconstrained adaptive optimizers, but produce final iterates that conform significantly better to the hierarchical structure of the TTN. When used for downstream compression tasks, deep networks optimized unconstrainedly break down completely, whereas deep TTNs trained with our methods react as expected.

The work is organized as follows: in Sec.~\ref{sec:preliminaries}, we formally introduce the TTN, and the geometric properties of its parameter and quotient space, as well as adaptive and distance-over-gradients schemes for general manifolds; in Sec.~\ref{sec:tool} 
we discuss the application of those to TTNs as our main contribution; 
in Sec.~\ref{sec:algorithms} we provide pseudocode and protocols for the implementation of the different SGD flavours; in Sec.~\ref{sec:experiments} we test the developed algorithms on computer vision benchmarks by using a CNN feature extractor and the TTN classifier; finally, in Sec.~\ref{sec:conclusions} we discuss the overall SGD picture for TTNs and their possible applications, open problems, and extensions for ML and beyond. 

\section{Preliminaries}\label{sec:preliminaries}
\subsection{Optimization on TTNs}
In this section, we will introduce general tree tensor networks, orthogonal TTN, and some important geometrical properties of TTN that are needed to develop Riemannian optimization algorithms. For a more detailed derivation and proofs see~\cite{Uschmajew:2013, Willner2025}.
\subsubsection{Tree tensor networks}
We will start with a definition of TTNs and of the notation that will be used throughout this work. Let $N=\{1,...,n\}$ be a set of labels for the \emph{physical} indices $i_1,...,i_n$ with dimensions $d_1,...,d_n$, that can be interpreted as the input of the tensor networks considered here. We will assume that $n$ is even for simplicity, but the definition can also be extended to an odd number of inputs. The tensor networks we consider here have an additional index $j$ with dimension $d_{\text{out}}$, which corresponds to the output.

A tree tensor network (TTN) is a tuple $\theta=(B^{(t)})_{t\in T}$ of order three \emph{core tensors} $B^{(t)}$, that are assembled through a rooted binary tree graph and labeled with elements $t\in T$. The set of labels $T$ is a subset of the power set $T\subset\mathcal{P}(N)$. Each tensor label is the set of all
labels of those physical indices that are children of the respective tensor. Due to the binary tree graph structure, each node $t\in T$ has two direct \emph{child nodes} $t_L, t_R\in T\cup N$ with $t=t_L\cup t_R$. 
Each core tensor connects three indices $v_{t_L}$, $v_{t_R}$ and $v_t$ of dimension $\chi_{t_L}$, $\chi_{t_R}$ and $\chi_T$, and it is canonnically identified with its matricization, meaning $B^{(t)}\in\R^{\chi_{t_L}\chi_{t_R}\times\chi_t}$. At the lowest layer of the TTN, each core tensor joins two consecutive physical indices and a single \emph{virtual} index $v_t$.
It holds that $v_{t_L} = i_{t_l}$, $v_{t_R} = i_{t_R}$, $\chi_{t_L} = d_{t_l}$, $\chi_{t_R} = d_{t_R}$. All other nodes, except for the root, join three virtual indices $v_{t_L}$, $v_{t_R}$ and $v_t$.
The root node $t = N$ joins the two final links, producing the outgoing index $v_t = j$, with dimension $\chi_t = d_\text{out}$. It will be explained in Section~\ref{sec: optimization} how TTNs may be employed as a machine learning model. Here we first focus on manifold properties of the TTN itself.

\subsubsection{Manifold structure}
\begin{figure}
    \centering
    \includegraphics[width=0.7\linewidth]{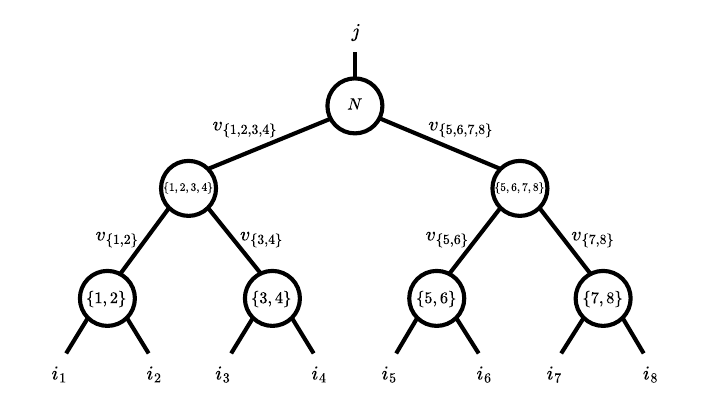}
    \caption{Graphical representation of a tree tensor network with $n=8$ physical indices $i_1,...,i_8$ and one "outgoing" index $j$. 
    The lowest layer of the TTN contains the nodes $\{1,2\},\{3,4\},\{5,6\},\{7,8\}$, that act on two consecutive physical indices each. The next layer are intermediate nodes $\{1,2,3,4\}, \{5,6,7,8\}$, followed by the root node with label $N=\{1,...,8\}$.}
    \label{fig:placeholder}
\end{figure}
The core tensors $\theta=(B^{(t)})_{t\in T}$ of the TTN parameterize the Euclidean space \cite{Uschmajew:2013}
\begin{equation*}
    \mathcal{E}:=\bigtimes_{t\in T}\R^{\chi_{t_L}\chi_{t_R}\times\chi_{t}}.
\end{equation*}
Furthermore, a TTN is always associated with a single tensor, which is obtained through the \emph{parameter map} $\tau:\mathcal{E}\rightarrow \R^{d_1\times\cdots\times d_n\times d_{\text{out}}}$ given by the contraction of all virtual indices $v_t$, $t\in T^*=T\setminus N$. The dimensions $(\chi_t)_{t \in T^*}$ of these virtual indices are referred to as the \emph{bond dimensions} of the TTN. They are subject to $\chi_t\leq\chi_{t_L}\chi_{t_R}$ \cite{Uschmajew:2013}, but may be chosen freely otherwise. Importantly, they govern the expressivity and computational complexity of the TTN as a model, and we assume them fixed throughout our considerations. A good initial guess for the bond dimensions can be found through the unsupervised construction algorithm \cite{Stoudenmire_2018}, they can be reduced using TTN compression algorithms \cite{scharf2026quantuminspiredrobustscalablesar} after training, or they may be chosen adaptively during optimization (see e.g. \cite{Ceruti2023}) in more advanced algorithms not covered in this work.

In practice, it is usually preferable to work with \emph{orthogonal} (also called isometrized) TTN, which define an important submanifold $\cT$ of general TTN. 
An orthogonal TTN $\theta\in\cT$ is defined in the same way as general TTN, with the additional restriction that all nodes $t\in T^*$ except for the root node are \emph{isometries}, i.e., elements of the \emph{Stiefel manifold} $B^{(t)}\in\St(\chi_{t_L}\chi_{t_R},\chi_t)$ satisfying
\begin{equation*}
    (B^{(t)})^T B^{(t)}=\mathbb{1}.
\end{equation*}
The parameter space of orthogonal TTN \cite{Willner2025} is given by
\begin{equation} \label{eq:orthogonal_ttn}
    \mathcal{T}=\left(\bigtimes_{t\in T^*}\St(\chi_{t_L}\chi_{t_R},\chi_t)\right)\times\R^{\chi_{N_L}\times\chi_{N_R}\times d_{\text{out}}}, 
\end{equation}
and its tangent space reads
\begin{equation*}
    T_\theta\cT = \left(\bigtimes_{t\in T^*}T_B\St(\chi_{t_L}\chi_{t_R},\chi_t)\right)\times\R^{\chi_{N_L}\chi_{N_R}\times d_{\text{out}}}.
\end{equation*}
It is the Cartesian product of the individual tangent spaces of the core tensors, where the tangent space of the Stiefel manifold $\St(n,k)$ is given by
\begin{equation*}
    T_X\St(n,k)=\{V\in\mathbb{R}^{n\times k}:X^TV+V^TX=0\}
\end{equation*}
and the tangent space of the root tensor is again a Euclidean space.

The space $\cT$ is an embedded submanifold of the Euclidean space $\mathcal{E}$. From the perspective of optimization, working with $\cT$ is preferable over $\mathcal{E}$ as it is on the one hand numerically more stable and on the other hand, we can exploit results of differential geometry to derive more sophisticated optimization algorithms. In the following, we will thus work exclusively with orthogonal TTN. 
Note that because of the \emph{gauge freedom}\footnote{This refers to the fact that at each virtual link one can insert an invertible matrix $A$ and its inverse $A^{-1}$, which can be used to alter the individual core tensors of the tensor network without affecting the result of the parameter map.}, the parameter map $\tau$ is not injective. Each TTN $\theta$ can be mapped to an orthogonal TTN $\theta'\in\cT$
through a process called \emph{isometrization}
, such that $\tau(\theta)=\tau(\theta')$. Therefore, considering only orthogonal TTN is not a restriction. 

\subsubsection{Quotient structure}\label{sec:quotientspace}

Note that even if we restrict the domain of $\tau$ to $\cT$, it is still not injective. This is due to the fact that at each virtual link $v_t$ of the TTN, one can insert an orthogonal matrix $A^{(t)}\in O(\chi_t)$ and its inverse $(A^{(t)})^T$, which can be contracted with the two core tensors connected by this link, respectively. After this contraction, the parameters of the two tensors changed, without breaking the isometry condition or changing the result of $\tau$ (see Fig.~\ref{fig:quotient}a for a graphical representation). Formally, this remaining gauge freedom is characterized by an action of the Lie group \cite{Uschmajew:2013}
\begin{equation}\label{eq: group G}
    \mathcal{G}=\{\mathcal{A}=(A^{(t)})_{t\in T^*}:A_t\in \mathrm O(\chi_t)\}.
\end{equation}
We can take the quotient of $\cT$ with respect to this Lie group to obtain the \emph{quotient manifold} $\cTG$ with associated quotient map $\pi:\cT\rightarrow\cTG$ that maps all parameters $\theta\in\mathcal{T}$ that are equivalent under $\mathcal{G}$ to the same representative $[\theta]\in\cTG$. 
Strictly speaking, in order for $\cTG$ to be a formal quotient manifold one has to restrict $\cT$ to only include tensors of full multilinear rank towards virtual links \cite{Uschmajew:2013, Willner2025}. As this does not affect our discussion, to lighten notation, we omit this technical detail. Moreover, in the following, whenever we write $\theta \in \cT$, we silently assume that $\theta$ fulfills those rank constraints.
If we push the parameter map to the quotient space $\hat\tau: \cTG\rightarrow\R^{d_1\times\cdots\times d_n\times d_{\text{out}}}$, satisfying $\tau = \hat \tau \circ \pi$, it finally becomes injective (see Fig.~\ref{fig:quotient}b) \cite{Willner2025}. This is desirable for optimization, however, it is important to note that the quotient space is an abstract space that is only useful in theory and the elements $[\theta]$ cannot directly be represented numerically. Nevertheless, it is still useful to study optimization on $\cTG$, which ultimately produces practical iteration rules in $\cT$.

\begin{figure}
    \centering
    \includegraphics[width=\linewidth]{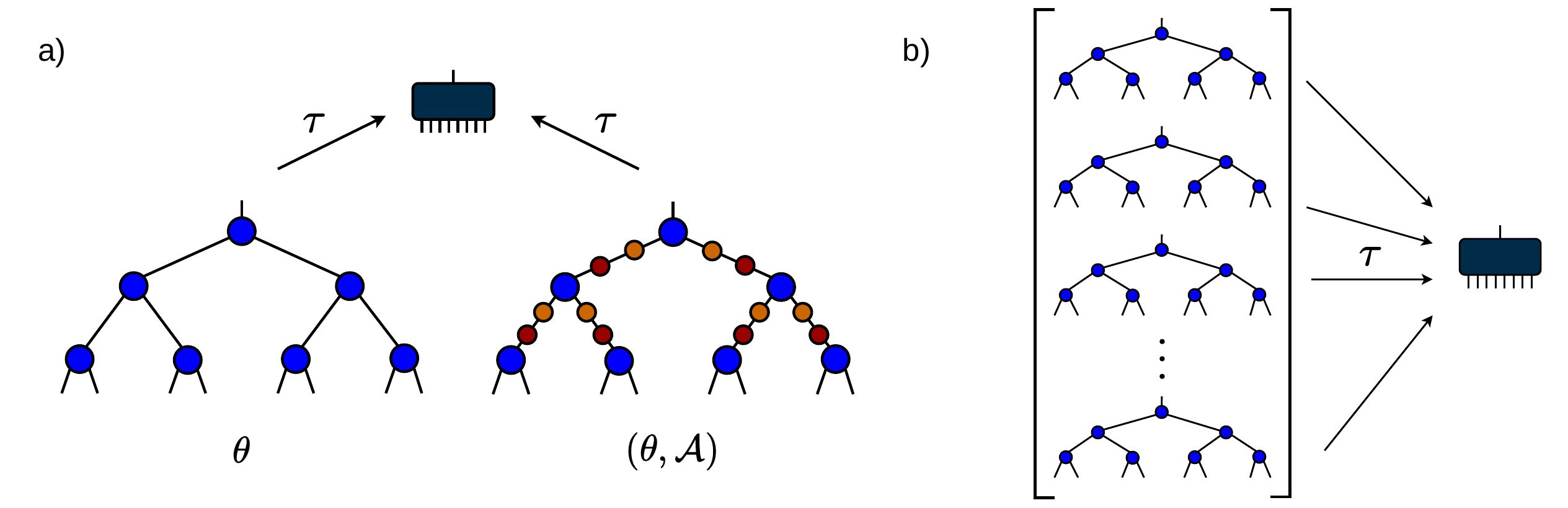}
    \caption{Illustration of the gauge freedom of orthogonal TTN and the corresponding quotient space. a) A given orthogonal TTN $\theta$ can be altered via the action of $\mathcal{A}\in \mathcal{G}$ by inserting pairs of orthogonal matrices $A^{(t)}$, $(A^{(t)})^T$ (red and orange, respectively) at each virtual link $v_t$, that change the original core tensors (blue) when contracted, without breaking the isometry constraint or changing the encoded full tensor $\Theta=\tau(\theta)$ (dark blue). b) All possible representatives $\theta\in\cT$ that are mapped to the same full tensor by $\tau$ are represented by a single abstract element $[\theta]\in\cTG$ in the quotient space.}
    \label{fig:quotient}
\end{figure}
Most importantly, we want to identify those directions in $T_{\theta}\cT$ that also lie in the tangent space of $T_{[\theta]}\cTG$, as those directions are invariant under the action of $\mathcal{G}$ and therefore influence $\Theta=\tau(\theta)$. The subspace containing the irrelevant directions is called the \emph{vertical space} $V_{\theta}\cT \subset T_{\theta}\cT$, whereas any complementary subspace in $T_{\theta}\cT$ is called the \emph{horizontal space} $H_{\theta}\cT \subset T_{\theta}\cT$. There are multiple valid choices for horizontal spaces of $\cT$, but we will only consider the simplest variant throughout, the \emph{Cartesian horizontal space}
\begin{equation} \label{eq:horizontal_space}
    H_\theta\cT = \left(\bigtimes_{t\in T^*}H_{B^{(t)}}\St(\chi_{t_L}\chi_{t_R},\chi_t)\right)\times\R^{\chi_{N_L}\chi_{N_R}\times d_{\text{out}}},
\end{equation}
that is again given by the Cartesian product of the individual horizontal spaces, where the canonical horizontal space of $\St(n,k)$ is given by
\begin{equation*}
    H_X\St(n,k)=\{V\in\R^{n\times k}:X^TV=0\}\,.
\end{equation*}

We can relate tangent vectors of the quotient space to vectors in the horizontal space through the \emph{horizontal lift} $\lift_{\theta}:T_{[\theta]}\cTG\rightarrow H_{\theta}\cT$. Furthermore, we equip $\cT$ with the metric
\begin{equation} \label{eq:euclidean_metric}
    \rho_\theta: T_\theta\cT \times T_\theta\cT \to \R, \quad \rho_\theta(\xi, \eta)= 
    \sum_{t \in T} \langle \delta B^{(t)} | \delta C^{(t)} \rangle =  
    \sum_{t \in T} \delta B^{(t)}_{v_{t_L}v_{t_R}v_{t}} \delta C^{(t)}_{v_{t_L}v_{t_R}v_{t}}, 
\end{equation}
where $\xi = (\delta B^{(t)})_{t \in T}$ and $\eta = (\delta C^{(t)})_{t \in T}$, respectively. Importantly, this also induces a Riemannian metric $\hat \rho$ on 
$T_{[\theta]}\cTG$ through horizontal lifts:
\begin{equation} \label{eq:quotient_metric}
    \hat \rho_{[\theta]}: T_{[\theta]}\cTG \times T_{[\theta]}\cTG \to \R, \quad \hat \rho_{[\theta]}(\hat \xi, \hat \eta) = \rho_\theta(\lift_\theta(\hat \xi), \lift_\theta(\hat \eta)).
\end{equation}
Thus both $(\cT, \rho)$ and $(\cTG, \hat \rho)$ are Riemannian manifolds.

\subsubsection{Projectors and retractions}\label{sec: Projectors and Retractions}
General vectors in Euclidean space $\mathcal{E}$ can be projected onto the tangent space and the horizontal space of $\cT$ with the orthogonal projectors $\Proj_{\theta}:\mathcal{E}\rightarrow T_{\theta}\cT$ and $\Proj^H_{\theta}:\mathcal{E}\rightarrow H_{\theta}\cT$, respectively. Both projectors leave the root tensor invariant and project the other core tensors to the respective spaces of the Stiefel manifold, with $P_X:\R^{n\times k}\rightarrow T_X\St(n,k)$ given by
\begin{equation*}
    P_{X}(V)=V-\frac{1}{2}X\left(X^TV+V^TX\right)
\end{equation*}
and $P_X^H:\R^{n\times k}\rightarrow H_X\St(n,k)$, given by
\begin{equation*}
    P_{X}^H(V)=(\mathbb{1}-XX^T)V.
\end{equation*}
Furthermore, we need a formalism to move from a point $\theta\in\cT$ to a new point $\theta'\in\cT$ along the direction of a tangent vector $\delta\theta\in T_{\theta}\cT$. This operation is called a \emph{retraction} $\mathrm R_{\theta}: T_{\theta}\cT\rightarrow\cT$, $(\theta,\delta\theta)\mapsto \mathrm R_{\theta}(\delta\theta)=\theta'$, and will be an important part of every iterative optimization algorithm. One possible retraction for $\cT$ we have already mentioned, namely the isometrization of TTN. After adding $\delta\theta$ to $\theta$ componentwise, the result is simply re-isometrized to project it back to $\cT$. Apart from this, retracting each tensor except for the root individually is also a retraction on $\cT$ since it is a Cartesian product of manifolds. Thus, we can use all retractions for the Stiefel manifold such as the Polar retraction $\mathrm R^{\text{polar}}_X:T_X\St(n,k)\rightarrow\St(n,k)$, $(X, \delta X)\mapsto UV^{T}$ with singular value decomposition of $X+\delta X=U\Sigma V^T$ or the QR-retraction $\mathrm R^{\text{QR}}_X:T_X\St(n,k)\rightarrow\St(n,k)$, $(X, \delta X)\mapsto Q$ with QR-decomposition of $X+\delta X=QR$.

A computationally viable tool to move tangent vectors from point $\theta \in \cT$ to a new point $\theta' \in \cT$ is \emph{vector transport}. A vector transport for tangents $\xi \in T_\theta\cT$ is given by
\begin{equation}\label{eq: projector and retraction full}
    \mathrm{V}^\cT_{\theta \to \theta'}: T_\theta\cT \to T_{\theta'}\cT, \quad\mathrm{V}_{\theta \to \theta'}(\xi) = \Proj_{\theta'}(\xi),
\end{equation}
and a suitable vector transport for $\xi \in H_\theta\cT$ reads
\begin{equation}\label{eq: projector and retraction quotient}
    \mathrm{V}^\cTG_{\theta \to \theta'}: H_\theta\cT \to H_{\theta'}\cT, \quad \mathrm{V}^H_{\theta \to \theta'}(\xi) = \Proj^H_{\theta'}(\xi).
\end{equation}
This second vector transport is also compatible with the quotient, i.e. $\DD \pi(\theta') \circ \mathrm{V}^\cTG_{\theta \to \theta'} \circ \lift_\theta$ is a vector transport on $\cTG$.
\subsubsection{Optimization function and Riemannian gradient}\label{sec: optimization}
The last ingredient we need is the Euclidean gradient and a mapping to the Riemannian gradient. In this work, we consider optimization functions $f:\cT\to\R$ of the form $f=\mathcal{L}\circ\tau$, with a loss function $\mathcal{L}:\R^{d_1\times\cdots\times d_n\times d_{\text{out}}}\to \R$. More specifically, we will do supervised learning following~\cite{stoudenmire2016, Novikov2017}, where a sample $x\in\R^n$ is encoded in a higher dimensional space through a \emph{feature map}
\begin{equation*}
    \phi(x) =  \phi^{(1)}(x_1)\otimes\cdots\otimes\phi^{(n)}(x_n),
\end{equation*}
that is composed of \emph{local} feature maps $\phi^{(i)}:\R\to\R^{d_i}$.
The model response is defined as
\begin{equation*}
    \phi^{(N)}_{j}=\tau(\theta)_{i_1,...,i_n,j}\, \phi(x)_{i_1,...,i_n}\,.
\end{equation*}
For a labeled data set $(x^k, y^k)_{k=1}^m$ and the squared error loss function, the optimization function takes the form $f(\theta) = \mathcal L(\tau(\theta)) = \tfrac{1}{m}\sum_k^m\|\Phi(x^k, \theta) - y^k\|^2$, where $\Phi(x, \theta) = \phi^{(N)} \in \R^{d_\text{out}}$.
To compute the Euclidean gradient, we can use the backpropagation algorithm that stores intermediate results in the \emph{forward propagation}, i.e., the computation of the loss, that are then reused in the gradient calculation which uses the chain rule to propagate the gradient back. 
In the forward propagation, we want to compute the response $\phi^{(N)}$ for a given TTN $\theta$ without actually constructing $\tau(\theta)$. To this end, we compute
\begin{equation}
    \phi^{(t)}_{v_t}=B^{(t)}_{v_{t_L}v_{t_R}v_t}\phi_{v_{t_L}}^{(t_L)}\phi_{v_{t_R}}^{(t_R)}
    \label{eq:forward}
\end{equation}
recursively $\forall t\in T$, starting from the leaf nodes (see Fig.~\ref{fig:backprop}a). All intermediate results $\phi^{(t)}$, $t\in T^*$ are stored as they are needed later for the backpropagation. The final result at the root node is equal to the response $\phi^{(N)}$. 
With this we can also compute the derivative of the loss with respect to the response
\begin{equation*}
    \bar\phi_j^{(N)}:=\frac{\partial\mathcal{L}}{\partial\phi^{(N)}_j}\,.
\end{equation*}
For the backpropagation, we can apply the chain rule to Eq.~\ref{eq:forward}, which yields
\begin{equation}
    \bar B^{(t)}_{v_{t_L}v_{t_R}v_t}:=\frac{\partial\mathcal{L}}{\partial B^{(t)}_{v_{t_L}v_{t_R}v_t}}=\bar\phi_{v_t}^{(t)}\phi_{v_{t_R}}^{(t_R)}\phi_{v_{t_L}}^{(t_L)}
    \label{eq:grad_B}
\end{equation}
for the TTN tensors and
\begin{align}
    \bar\phi_{v_{t_L}}^{(t_L)}:=&\frac{\partial\mathcal{L}}{\partial\phi^{(t_L)}_{v_{t_L}}}=B^{(t)}_{v_{t_L}v_{t_R}v_t}\bar\phi_{v_t}^{(t)}\phi_{v_{t_R}}^{(t_R)}\\
    \bar\phi_{v_{t_R}}^{(t_R)}:=&\frac{\partial\mathcal{L}}{\partial\phi^{(t_R)}_{v_{t_R}}}=B^{(t)}_{v_{t_L}v_{t_R}v_t}\bar\phi_{v_t}^{(t)}\phi_{v_{t_L}}^{(t_L)}\, .
    \label{eq:grad_phi}
\end{align}
for the intermediate results, which now need to be computed in reverse order starting from the root node (see Fig.~\ref{fig:backprop}b). The derivatives of the input $\bar\phi^{(i)}$, $i = 1, \ldots, n$ are not needed and can be skipped. The full Euclidean gradient is then given by assembling the individual derivatives
\begin{equation*}
    \nabla  f (\theta) =(\bar B^{(t)})_{t\in T}\in \mathcal{E}.
\end{equation*}
\begin{figure}
    \centering
    \includegraphics[width=0.7\linewidth]{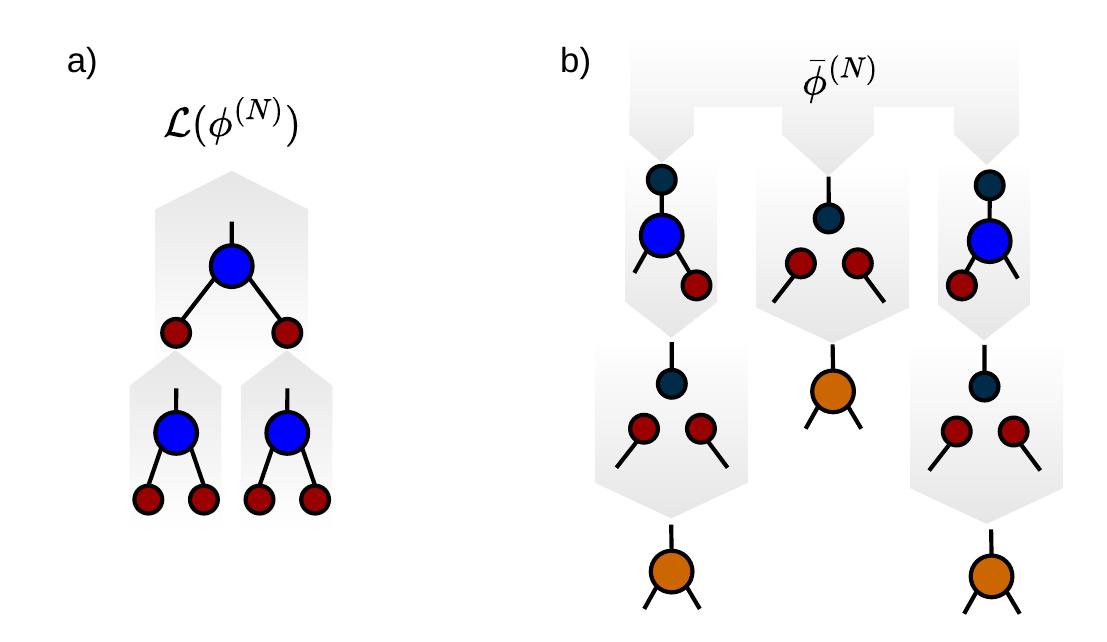}
    \caption{a) Forward propagation. At each node $t$, the corresponding features $\phi^{(t_L)}$ and $\phi^{(t_R)}$ (here in red) are contracted with the tensor $B^{(t)}$ (in blue), which results in a new feature vector $\phi^{(t)}$ (Eq.~\ref{eq:forward}). This feature vector serves as one of the inputs for the parent node in the next layer, indicated by the gray arrow. The final result of the root node $\phi^{(N)}$ (the response) is used to compute the loss $\mathcal{L}$.  b) Backpropagation. The derivative of the loss w.r.t. the response $\bar\phi^{(N)}$ (depicted in dark blue) is inserted into Eq.~\ref{eq:grad_B}-\ref{eq:grad_phi} to compute $\bar\phi^{(t_L)}$ (left), $\bar B^{(t)}$ (center, resulting in the orange tensor) and $\bar\phi^{(t_R)}$ (right), where the intermediate results (red) are reused from the forward propagation. $\bar\phi^{(t_L)}$ and $\bar\phi^{(t_R)}$ are propagated to the nodes in the lower layer where they are needed for the computation of $\bar B^{(t)}$. The three orange tensors constitute the Euclidean gradient $\nabla f$.}
    \label{fig:backprop}
\end{figure}
Note, that as an alternative to Equations~(\ref{eq:grad_B}–\ref{eq:grad_phi}), one can also use automatic differentiation to compute the Euclidean gradient.

Finally, the Riemannian gradient on $\cT$ is given by the projection of the Euclidean gradient $\nabla f$ to the tangent space, since $\cT$ is a Riemannian submanifold of $\mathcal{E}$: $\grad f(\theta)=\Proj_{\theta}(\nabla f(\theta))$. The Riemannian gradient $\grad f(\theta)$ is already the right update direction for Riemannian gradient descent formulated on $\cT$, however it does not adequately represent the Riemannian gradient of the quotient space $\cTG$ with respect to metric $\hat \rho$. This quotient gradient lifted to $\cT$ is actually equal to the projection of $\grad f$ to the horizontal space
\begin{equation} \label{eq:lifted_gradient}
    \lift_{\theta}(\grad \hat f([\theta])) = \Proj^{H}_{\theta}(\grad f(\theta)) = \Proj_{\theta}^H(\nabla f(\theta)),    
\end{equation}
with $\hat f: \cTG \rightarrow \R, \, \hat f = \mathcal L\circ \hat \tau$.
The second identity holds since $H_{\theta}\cT\subset T_{\theta}\cT$. Optimizing according to update directions~\eqref{eq:lifted_gradient} on $\cT$ implicitly represents an optimization in $\cTG$ \cite{Willner2025}.

\subsection{Adaptive Riemannian optimization} \label{sec:radaopt}
Adaptive schemes in optimization on Euclidean spaces work by accumulating information of past iterates to coordinate-wisely rescale the gradient of the current iterate. Given an optimization function $f: \R^n \to \R$ and an initial iterate
$w_0 = (w^{(1)}_0, \ldots, w^{(n)}_0) \in \R^n$, adaptive schemes iterate according to
\begin{equation*}
    w^{(i)}_{k+1} = w^{(i)}_k - \alpha h(g^{(i)}_0, \ldots, g^{(i)}_k) \cdot g_k^{(i)} 
\end{equation*}
where $g_k = \nabla f(w_k)$ and $\alpha$ is some step-size. While $-\alpha g_k^{(i)}$ represents the plain standard gradient descent update direction, the real-valued function $h$ introduces a coordinate-wise rescaling over past gradients. 
More advanced adaptive algorithms such as ADAM~\cite{kingma2015} 
also make use of stabilized update directions, replacing $-\alpha g_k^{(i)} $ with $-\alpha m_k^{(i)}$,
where $m_k = \beta m_{k-1} + (1 - \beta)g_k$ is the momentum for some decay parameter $\beta \in (0, 1)$.  Those two strategies can yield enormous improvements over standard gradient descent, in particular in stochastic optimization settings where $g_k$ are gradient estimates that are sparse but possibly instable. 

Extending this mechanism to Riemannian optimization is no trivial feat, since manifolds do not come with an intrinsic coordinate representation. Here we focus on the techniques presented in~\cite{Becigneul2018}, which apply to the case of a Riemannian manifold $(\mathcal{M}, \rho)$, that factors into a Cartesian product of Riemannian manifolds $(\mathcal{M}^{(i)}, \rho^{(i)})$:
\begin{equation}\label{eq:cartesian_manifold}
    \mathcal{M} = \bigtimes_{i=1}^n \mathcal{M}^{(i)}, \quad \rho = \sum_{i=1}^n \rho^{(i)}.
\end{equation}
Points $w \in \mathcal{M}$ may be written as $w = (w^{(1)}, \ldots, w^{(n)})$, where $w^{(i)} \in \mathcal{M}^{(i)}$ and tangent vectors $\delta w \in T_w\mathcal{M}$ separate as $\delta w = (\delta w^{(1)}, \ldots, \delta w^{(n)})$, where $\delta w^{(i)} \in T_{w^{(i)}}\mathcal{M}^{(i)}$. The individual components serve as block coordinates in adaptive Riemannian schemes, so one iterates
\begin{equation*}
    w_{k+1} = \mathrm R_{w_k}\left((- \alpha h(g^{(i)}_0, \ldots, g^{(i)}_k) \cdot g_k^{(i)} )_{i=1, \ldots, n}\right), 
\end{equation*}
where $g_k = (g_k^{(1)},\ldots, g_k^{(n)}) = \grad f(w_k)$. To ensure compatibility with the manifold structure, the update direction goes through a retraction $\mathrm R$ on $\mathcal{M}$, and of course $h$ has to be adapted to act on block-coordinates instead. 
This framework additionally allows the incorporation of momentum-based updates by employing a vector transport $\mathrm{V}$ on $\mathcal{M}$ via
\begin{equation*}
    m_k = \beta\,\mathrm{V}_{w_{k-1} \to w_k}(m_{k-1}) + (1 - \beta) g_k,
\end{equation*}
which is commonly used in Riemannian versions of ADAM~\cite{Becigneul2018}.

Notably, the above block-coordinate formulas also captures the MUON optimizer~\cite{Keller2024}, which is an adaptive optimizer for the hidden layers of neural networks. It comprehends each weight matrix $W^{(i)}$ and corresponding gradient $G^{(i)}$ as a separate block-coordinate, effectively imposing a Cartesian product of matrix spaces $\mathcal M^{(i)} = \R^{\chi^{\text{in}}_i \times \chi^{\text{out}}_i}$, where
$(\chi^{\text{in}}_i, \chi^{\text{out}}_i)$ is the dimension of the layer $i$. MUON introduces adaptiveness by orthogonalizing each update matrix $G^{(i)}$ individually using the polar decomposition. In its most basic form it iterates according to
\begin{equation*}
    W^{(i)}_{k+1} = W^{(i)}_{k} - \alpha\, \mathrm{qf}(G^{(i)}_k),
\end{equation*}
where $\mathrm{qf}(\cdot)$ refers to the orthogonal factor of the polar decomposition. This update formula manages without a retraction because the considered matrix spaces are Euclidean.

\subsection{Riemannian distance over gradients}
Tuning the learning rate $\alpha$ for stochastic optimizers to a given machine learning task can represent a considerable overhead, both when developing and deploying models. Distance over gradients (DOG)~\cite{Ivgi2023} provides a tuning-free dynamic step-size formula for stochastic gradient descent and it was extended to optimization on manifolds in~\cite{Dodd2024}. Concretely, on a Riemannian manifold $(\mathcal{M}, \rho)$ with sectional curvature $\kappa$, a step-size schedule is given by
\begin{equation} \label{eq:dog}
    \alpha_k = \frac{r_k}{\sqrt{\zeta_\kappa(r_k) \sum_{{k'}=0}^k ||g_{k'}||_{\rho}^2}}, \quad r_k = \max_{{k'} < k} d(w_0, w_{k'})
\end{equation}
where $g_k = \grad f(w_k)$. The expression $d(x_0, x_k)$ is the \emph{geodesic distance} between the initial and $k$-th iterate. The \emph{geometric curvature function} $\zeta_\kappa$ is included to accommodate for manifolds of negative sectional curvature and it reads
\begin{equation*}
    \zeta_\kappa(r) = \begin{cases}
    \frac{r \sqrt{\abs{\kappa}}}{\mathrm{tanh}(r\sqrt{\abs{\kappa}})} &\text{if } \kappa < 0, \\
    1 & \text{otherwise.}
    \end{cases}
\end{equation*}
Formula~\eqref{eq:dog} is applicable to stochastic Riemannian gradient descent $x_{k+1} = \mathrm R_{x_k}(- \alpha_k g_k)$ of a geodesically convex function $f$.
In practical application, the numerator of~\eqref{eq:dog} is taken at least $\varepsilon \in (0,1)$, which represents an initial estimate of the step-size and is used to avoid step-sizes of magnitude zero. While this may seem like an additional hyperparameter, numerical experiments show that DOG-like algorithms react very insensitively to the choice of $\varepsilon$~\cite{Dodd2024}.
Importantly, DOG requires the evaluation of the geodesic distance defined as
\begin{equation*}
    d(x, y) =  \min_\gamma L(\gamma) = \min_{\gamma}\int_a^b \rho_{\gamma(s)}(\gamma'(s), \gamma'(s))^{\frac{1}{2}} \dd s,
\end{equation*}
which is the minimum length of all curves $\gamma:[a, b] \to \mathcal{M}$ connecting
$\gamma(a) = x$ and $\gamma(b) = y$, and the minimizer will be a geodesic on $\mathcal M$. We quickly recapitulate results that exists for the Stiefel and Grassmann manifolds. 
A closed form solution for geodesic distances on $\St(n, k)$ does not exist, but~\cite{Mataigne2026} provides a lower bound. For matrices $X, Y \in \St(n, k)$ it holds that 
\begin{equation} \label{eq:stiefel_dist_bound}
    d^\St(X, Y) \geq q_k(\|X-Y\|_F) := 2 \sqrt{k}\arcsin\left(\frac{\|X-Y\|_F}{2\sqrt{k}}\right).
\end{equation}
The sectional curvature of the Stiefel manifold is bounded from below by $-\tfrac{1}{2}$~\cite[Thm. 10]{Zimmermann2025}.

The Grassmann manifold forms as the quotient $\Gr(n, k) = \St(n, k)/\mathrm{O}(k)$ of the Stiefel manifold with the orthogonal group. It hence contains equivalence classes $[X] = \{ X Q \colon Q \in \mathrm{O}(k) \}$ of matrices $X \in \St(n, k)$ whose columns form an orthonormal basis for the same $k$-dimensional subspace of $\R^n$. An expression for the geodesic distance ~\cite[Section 5.1]{Bendokat2024} in terms of Stiefel representatives $X, Y \in \St(n, k)$ of $[X], [Y] \in \Gr(n,k)$ reads
\begin{equation*}
    d^\Gr([X], [Y]) = \sqrt{\sum_i \arccos(\sigma_i(X^TY))^2}.
\end{equation*}
Here, $\sigma_i$ refers to the $i$-th singular value. The sectional curvature of the Grassmann manifold is non-negative~\cite[Prop. 12]{Bendokat2024}.

\section{Tools}\label{sec:tool}
\subsection{Adaptive optimization of TTNs}

The developments of Section~\ref{sec:radaopt} directly apply to optimization on the TTN-manifold $\cT$. Equations~\eqref{eq:orthogonal_ttn} and~\eqref{eq:euclidean_metric} together attest that $(\cT, \rho)$ fits into the framework~\eqref{eq:cartesian_manifold}, where the individual core tensors $B^{(t)}$ of the TTN  $\theta = (B^{(t)})_{t \in T}$ are comprehended as block coordinates for adaptive Riemannian schemes. Tangent vectors $\xi = (\delta B^{(t)})_{t \in T}$, $\eta = (\delta C^{(t)})_{t \in T}$ in 
$T_\theta \cT$ also represent block-coordinate variations and the metric $\rho$ acts separably on those variations.

Optimization on the quotient $\cTG$ however does not immediately comply with the proposed mechanism of adaptivity, as $\cTG$ intrinsically couples all nodes of the network and does not factor into a Cartesian product of manifolds~\cite{Willner2025}. Still any element $[\theta]\in T_\theta\cTG$ can be identified with a representative $\theta = (B^{(t)})_{t \in T}$, that reads as a tuple of block coordinates. Whats more, both the quotient tangents and the quotient metric can be lifted uniquely to the total space, where they decompose as tensor-wise expressions: 
\begin{equation*}
    \lift_\theta(\hat \xi) = (\delta B^{(t)})_{t \in T} \in H_\theta\cT, \quad \hat \rho_{[\theta]}(\hat \xi, \hat \eta) = 
    \rho_\theta(\lift_\theta(\hat \xi), \lift_\theta(\hat \eta)) = 
    \sum_{t \in T} \langle \delta B^{(t)} | \delta C^{(t)} \rangle.
\end{equation*}
Thus, the practical application of adaptive schemes on $\cTG$ is just as straightforward as it is on~$\cT$, since optimization on $\cTG$ happens only implicitly anyway through representatives $\theta \in \cT$ and horizontal lifts of the quotient gradient~\eqref{eq:lifted_gradient}. 

In the next section we will present several algorithms, like Riemannian ADAM or Riemannian MUON, that employ adaptive schemes in the quotient optimization setting, to which theoretical convergence guarantees like \cite[Theorem 1 \& 2]{Becigneul2018} can be extended. In order to highlight how this may be done, we show convergence of Riemannian ADAGRAD in Appendix~\ref{sec:radagrad_conv}. Being one of the simplest adaptive algorithms, this allows us to put particular focus on the peculiarities involving the TTN quotient.

\subsection{Distance over gradients on TTN-manifolds}
Both for optimization on $\cT$ and optimization $\cTG$, the main obstacle to achieve DOG-like algorithmic is the computation of the geodesic distance.

We begin our considerations for $\cT$, where we again use that it factorizes as a Cartesian product of manifolds equipped with a product metric. For two points $\theta = (B^{(t)})_{t \in T}$ and $\theta' = (C^{(t)})_{t \in T}$ it holds that
\begin{equation} \label{eq:geodesic_distance_T}
    d^{\cT}(\theta, \theta')^2 = \sum_{t \in T^*} d^\St(B^{(t)}, C^{(t)})^2 + ||B^{(N)} - C^{(N)}||^2_F,
\end{equation}
so using~\eqref{eq:stiefel_dist_bound}, the geodesic distance on $\cT$ can be bounded by
\begin{equation*}
    d^{\cT}(\theta, \theta') \geq \sqrt{\sum_{t \in T^*} q_{k_t}(||B^{(t)} - C^{(t)}||_F)^2 + ||B^{(N)} - C^{(N)}||^2_F} \,,
\end{equation*}
which we will use as a proxy for DOG-like algorithms. Compared to standard DOG, this will yield smaller, more conservative step-sizes, because we are underestimating $r_k$ in~\eqref{eq:dog}.

Concerning geodesic distances on $\cTG$, 
for two points $[\theta], [\theta'] \in \cTG$ we consider a connecting geodesic $\hat \gamma:[0, 1] \to\cTG$, as well as its horizontal lift $\gamma$ to $\cT$, that satisfies $\gamma'(s) \in H_{\gamma(s)}\cT$ for all $s \in [0, 1]$. By writing $\gamma(0) = \theta = (B^{(t)})_{t \in T}$, $\gamma(1) = \theta' = (C^{(t)})_{t \in T}$ it holds that
\begin{equation} \label{eq:quotient_geodesic_distance}
    d^\cTG([\theta], [\theta'])^2 = \sum_{t \in T^*} 
    d^\Gr([B^{(t)}], [C^{(t)}])^2
    + ||B^{(N)} - C^{(N)}||^2_F\, .
\end{equation}
For a detailed derivation, see Appendix~\ref{sec:geodesic_distance_quotient}, where we use the fact that $\cTG$ shares its geodesic spray with a Cartesian product of Grassmann manifolds
\begin{equation} \label{eq:grassmann_product}
    \cZ = \left(\bigtimes_{t \in T^*} \Gr(\chi_{t_L}\chi_{t_R},\chi_t)\right) \times \R^{\chi_{N_L}\chi_{N_R}\times d_{\text{out}}}.
\end{equation}
There is still a caveat for the practical application of~\eqref{eq:quotient_geodesic_distance}. It only holds for a specific representation $\theta'\in \cT$ of $[\theta'] \in \cTG$, concretely for $\theta' = \gamma(1)$, the endpoint of the horizontal lift of the geodesic connecting $[\theta]$ and $[\theta']$, to which we do not have access in general.
Indeed, picking a different representative $\tilde \theta'$ with $\pi(\tilde \theta') = [\theta']$ will in general produce a different result when evaluating~\eqref{eq:quotient_geodesic_distance}.
Note however that the application of DOG algorithms~\eqref{eq:dog} only requires the geodesic distances between iterates, which transform into each other through repeated application of retractions. Using the exponential map~\eqref{eq:exp} on $\cTG$ as a retraction ensures that iterates remain in compatible representations. On the other hand, any other retraction produces first-order approximations of geodesics on the quotient, so it is safe to assume that the representations we obtain do not deviate far from their "correct" geodesic representation. Thus, equation~\eqref{eq:quotient_geodesic_distance} still applies as an approximation for arbitrary retractions.

\begin{figure}[h]
    \centering
    \includegraphics[width=\linewidth]{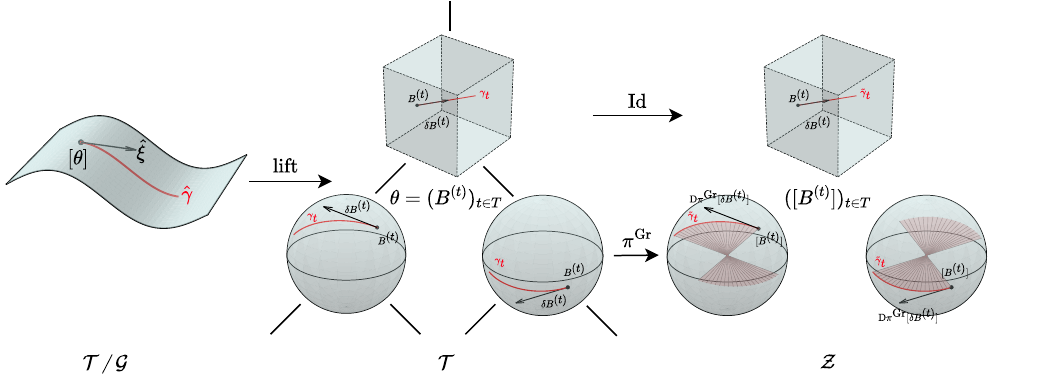}
    \caption{Even though $\cTG$ does not factor, iterates $[\theta]$ and update directions $\hat \xi$ can be lifted to a point $\theta = (B^{(t)})_{t \in T} \in \cT$ and a direction $\xi = (B^{(t)})_{t \in T} \in H_\theta\cT$ that factors. Furthermore, geodesics $\hat \gamma \in \cTG$ ascend to horizontal paths in $(\gamma_t)_{t \in T}\in\cT$ which, at non-root nodes, descend to geodesics $\tilde \gamma_t \in \Gr(\chi_{t_L}\chi_{t_R},\chi_t)$. No quotient is taken at the root node, which resides in a Euclidean space, and where geodesics simplify to straight lines anyway.}
    \label{fig:quotient_chain}
\end{figure}

\section{Algorithms}\label{sec:algorithms}
\subsection{RADAM}
The Riemannian ADAM algorithm (RADAM) for TTNs is presented as Algorithm~\ref{alg:radam} and obtained as an application of~\cite[Alg. 1]{Becigneul2018} to TTNs. Importantly, it covers both optimization on the total space $\cT$, as well as optimization on the quotient $\cTG$. Optimization on the quotient happens only implicitly by using representatives and update directions from the total space and is achieved by picking the tuple of projector and vector transport as $(\mathrm P,\mathrm V) = (\Proj^H, \mathrm V^\cTG)$ according to~\eqref{eq: projector and retraction quotient}. By instead picking $(\mathrm P,\mathrm V) = (\Proj, \mathrm V^\cT)$ as in~\eqref{eq: projector and retraction full}, the quotient structure is ignored and the iteration happens entirely in the total~space. Here, $f_k$ are function oracles of function $f$ given by evaluating the loss only on a random batch of training data.

\begin{figure}[ht]
\centering
\begin{minipage}[t]{.49\textwidth}
\begin{algorithm}[H]
\caption{RADAM for TTNs} \label{alg:radam}
\begin{algorithmic}
    \Require $\theta_1 = (B^{(t)}_1)_{t \in T} \in \cT$, learning rate $\alpha$, decay parameters $\beta_1, \beta_2 \in (0, 1)$, projector P\,, retraction R, vector transport V
    \State Set $m_0 = (0)_{t \in T}, v_0 = (0)_{t \in T}$
    \For{$k = 1, 2, \ldots$}
    \State $(G_k^{(t)})_{t \in T} = \mathrm P (\nabla  f_k(\theta_k))$
    \For{$t \in T$}
        \State $m_k^{(t)} = \beta_1 m_{k-1}^{(t)} + (1 - \beta_1) G_k^{(t)}$
        \State $v_k^{(t)} = \beta_2 v_{k-1}^{(t)} + (1 - \beta_1) \langle G_k^{(t)} \mid G_k^{(t)}\rangle$
        \State $\delta B^{(t)}_{k} = -\alpha\, m_k^{(t)}/\sqrt{v_k^{(t)}}$
    \EndFor
    \State $\theta_{k+1} = \mathrm R_{\theta_k}((\delta B^{(t)}_{k})_{t \in T}) $
    \State $m_{k+1} = \mathrm{V}_{\theta_k \to \theta_{k+1}}(m_k)$
    \EndFor
    \vspace{3mm}
\end{algorithmic}
\end{algorithm}
\end{minipage}
\hfill
\begin{minipage}[t]{.49\textwidth}
\begin{algorithm}[H]
\caption{RDOG for TTNs} \label{alg:rdog}
\begin{algorithmic}
    \Require $\theta_1 = (B^{(t)}_1)_{t \in T} \in \cT$, step-size estimate $\varepsilon$,
    projector P\,, retraction R, geodesic distances $(d^{(t)})_{t \in T}$
    \State Set $r_0 = (\varepsilon)_{t \in T}$, $n_0 = (0)_{t \in T}$, 
    \For{$k = 1, 2, \ldots$}
    \State $(G_k^{(t)})_{t \in T} = \mathrm P (\nabla  f_k(\theta_k))$
    \For{$t \in T$}
        \State $n_{k+1}^{(t)} = n_k^{(t)} + \langle G_k^{(t)} \mid G_k^{(t)}\rangle$
        \State $\delta B^{(t)}_{k} = - r_k^{(t)} G_k^{(t)}/\sqrt{\zeta_\kappa(r_k^{(t)})n_k^{(t)}}$
    \EndFor
    \State $\theta_{k+1} = \mathrm R_{\theta_k}((\delta B^{(t)}_{k})_{t \in T}) $
    \For{$t \in T$}
        \State $r_{k+1}^{(t)} = \max \{r_k^{(t)}, d^{(t)}(B^{(t)}_{0}, B^{(t)}_{k})\}$
    \EndFor
    \EndFor
\end{algorithmic}
\end{algorithm}
\end{minipage}
\end{figure}

\subsection{RDOG}
Algorithm~\ref{alg:rdog} presents an adaptive version of Riemannian distance over gradients (RDOG) for TTNs. It is similar to layer-wise DOG~\cite[Section 4]{Ivgi2023}, in that it keeps a separate tracker of distance $r^{(t)}_k$ for each core tensor $t \in T$, as opposed to a single distance for the whole tensor network. Since both for~\eqref{eq:geodesic_distance_T} and~\eqref{eq:quotient_geodesic_distance}, the total geodesic distances form as the norm of node-wise distances, adaptive RDOG employs exactly those node-wise distances denoted by $(d^{(t)})_{t \in T}$. For our experiments, we pick $d^{(N)}(B^{(t)}, C^{(t)}) = \|B^{(t)}-C^{(t)}\|_F$ at the root tensor, but differ in
\begin{equation} \label{eq:node_distance}
    d^{(t)}(B^{(t)},C^{(t)}) = 
    \begin{cases}
        d^\Gr(\pi^\Gr(B^{(t)}), \pi^\Gr(C^{(t)})) &\text{ on }\cTG,\\
        q_{k_t}(\|B^{(t)}-C^{(t)}\|_F) &\text{ on }\cT,
    \end{cases}
    \qquad\text{for all } t\in T^*.
\end{equation}
Like for RADAM, the projector P is chosen accordingly, depending on whether we optimize on $\cT$ or $\cTG$.

\subsection{RMUON}
We present two Riemannian versions of the MUON optimizer for TTNs. Algorithm~\ref{alg:rmuon} combines MUON updates with momentum, whereas Algorithm~\ref{alg:rmuondog} feeds MUON updates into a DOG-like scheme. 

Recall that MUON takes steps according to orthogonal parts of gradient components. Since $\cT$ forms as a Cartesian product of matrix manifolds, that is,~all tensors in a tree tensor network can canonically be reshaped into a matrices,  node-wise gradients may be orthogonalized according to the polar decomposition to mimic MUON behavior. These orthogonalized directions are even compatible with the quotient $\cTG$: As seen in \cite{Li2026} for any $\delta X \in H_X\St(n,k)$, it holds that $\mathrm{qf}(\delta X) \in H_X\St(n,k)$ as well. Thus, when picking $(\mathrm P,\mathrm V) = (\Proj^H, \mathrm V^\cTG)$ according to~\eqref{eq: projector and retraction quotient}, all update steps of Algorithms~\ref{alg:rmuon} \&~\ref{alg:rmuondog} will reside in $H_{\theta_k}\cT$ and descend to well-formed updates in $T_{[\theta_k]}\cTG$. Again, by employing $(\mathrm P,\mathrm V) = (\Proj, \mathrm V^\cT)$ as in~\eqref{eq: projector and retraction full} instead, the iteration happens entirely in $\cT$, ignoring the quotient.

\begin{figure}[ht]
\centering
\begin{minipage}[t]{.49\textwidth}
\begin{algorithm}[H]
\caption{RMUON for TTNs} \label{alg:rmuon}
\begin{algorithmic}
    \Require $\theta_1 = (B^{(t)}_1)_{t \in T} \in \cT$, learning rate $\alpha$, decay parameter $\beta_1 \in (0, 1)$, projector P\,, retraction R, vector transport V
    \State Set $m_0 = (0)_{t \in T}$
    \For{$k = 1, 2, \ldots$}
    \State $(G_k^{(t)})_{t \in T} = \mathrm P (\nabla  f_k(\theta_k))$
    \For{$t \in T$}
        \State $m_k^{(t)} = \beta_1 m_{k-1}^{(t)} + (1 - \beta_1) G_k^{(t)}$ 
        \State $\delta B^{(t)}_{k} = -\alpha\, \mathrm{qf}(m_k^{(t)})$
    \EndFor
    \State $\theta_{k+1} = \mathrm R_{\theta_k}((\delta B^{(t)}_{k})_{t \in T}) $
    \State $m_k = \mathrm{V}_{\theta_k \to \theta_{k+1}}(m_k)$
    \EndFor
\end{algorithmic}
\end{algorithm}
\end{minipage}
\hfill
\begin{minipage}[t]{.49\textwidth}
\begin{algorithm}[H]
\caption{RMUONDOG for TTNs} \label{alg:rmuondog}
\begin{algorithmic}
    \Require $\theta_1 = (B^{(t)}_1)_{t \in T} \in \cT$, step-size estimate $\varepsilon$,
    projector P\,, retraction R, geodesic distances $(d^{(t)})_{t \in T}$
    \State Set $r_0 = (\varepsilon)_{t \in T}$, $n_0 = (0)_{t \in T}$
    \For{$k = 1, 2, \ldots$}
    \State $(G_k^{(t)})_{t \in T} = \mathrm P (\nabla  f_k(\theta_k))$
    \For{$t \in T$}
        \State $n_{k+1}^{(t)} = n_k^{(t)} + \langle \mathrm{qf}(G_k^{(t)}) \mid \mathrm{qf}(G_k^{(t)})\rangle$
        \State $\delta B^{(t)}_{k} = - r_k^{(t)} \mathrm{qf}(G_k^{(t)})/\sqrt{\zeta_\kappa(r_k^{(t)})n_k^{(t)}}$
    \EndFor
    \State $\theta_{k+1} = \mathrm R_{\theta_k}((\delta B^{(t)}_{k})_{t \in T}) $
    \For{$t \in T$}
        \State $r_{k+1}^{(t)} = \max \{r_k^{(t)}, d^{(t)}(B^{(t)}_{0}, B^{(t)}_{k})\}$
    \EndFor
    \EndFor
\end{algorithmic}
\end{algorithm}
\end{minipage}
\end{figure}

\section{Numerical Experiments}\label{sec:experiments}

To numerically evaluate the proposed methods, they are implemented using the \texttt{pytorch} framework, which allows for automatic differentiation to calculate Euclidean gradients $\nabla f_k(\theta)$ of the objective function on random mini-batches. The objective $f = \mathcal L \circ \tau$ is given by the squared error loss and the batch size is fixed to 256 throughout our considerations.
Initial parameters $\theta_1 \in \cT$ were found using the unsupervised construction algorithm introduced in \cite{Stoudenmire_2018}. On a given architecture, the same initial iterate was used for all optimizers. Learning rates/ step-size estimates were found by conducting a simple grid-search as summarized in the following table, whereas momentum-related hyperparameters remain fixed at $\beta_1=0.9$, $\beta_2=0.999$.
\begin{center}
\begin{tabular}{c|c|c|c|c|c}
    dataset & $\alpha^{\text{ADAM}}$ & $\alpha^{\text{RADAM}}$ & $\alpha^{\text{RMUON}}$ & $\varepsilon^{\text{RDOG}}$ & $\varepsilon^{\text{RMUONDOG}}$ \\ \hline
    Fashion MNIST & 0.001 & 0.03 & 0.007 & 0.1 & 0.1  \\
    CIFAR10 & 0.001 & 0.02 & 0.005 & 0.1 & 0.05  \\
    Imagenette &  0.001 & 0.02 & 0.005 & 0.02 & 0.05 \\
\end{tabular}
\end{center}
For Riemannian optimizers, we componentwisely use the recently proposed POGO algorithm \cite{Javaloy2026} as retraction. It can be understood as an approximation to the polar retraction, that does not enforce strict orthogonality every step, but instead remains close to the orthogonal manifold at all times, in favor of cheaper and more GPU-friendly iteration rules. Furthermore, we omit a detailed discussion of Riemannian algorithms that ignore the quotient structure: As seen in Fig. \ref{fig:cifar10T} in the appendix, they tend to perform very similarly but slightly worse than their quotient counterparts. All experiments were conducted on a \texttt{NVIDIA A30} GPU in single-precision.

\subsection{Fashion MNIST}
As a first test, we use the Fashion MNIST dataset \cite{fashion} that contains 60000 training- and 10000 test images with $28\times28$ pixels and one channel each. As in \cite{Hao2024,Klug2026}, we compress all images to $16\times16$ pixels to reduce the number of features of the TTN model, as we will use each pixel directly as an input feature. The feature map we use to encode the pixels is given by \cite{Klug2026}
\begin{equation*}
    \phi^{(i)}(x_i)=\frac{1}{\sqrt{1+x_i^2}}\begin{pmatrix}1\\x_i\end{pmatrix},
\end{equation*}
meaning that $n=256$ and $d = d_1,...,d_n=2$. For the Fashion MNIST dataset, we choose a maximum bond dimension of $\chi_{\text{max}}=16$.
\begin{figure}
    \centering
    \begin{minipage}{0.5225\textwidth}
        \centering
        \includegraphics[width=\textwidth]{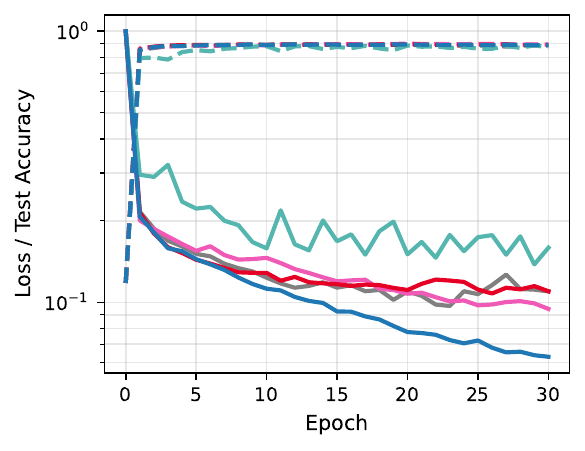}
    \end{minipage}\hfill
    \begin{minipage}{0.4775\textwidth}
        \centering
        \includegraphics[width=\textwidth]{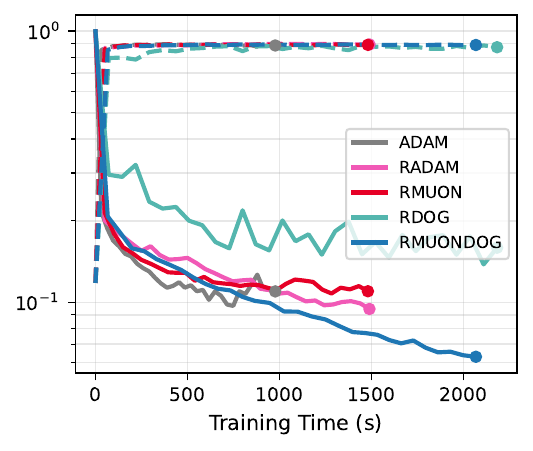}
    \end{minipage}
    \caption{Comparison of the different Riemannian optimizers and the conventional ADAM optimizer through loss (solid) and accuracy (dashed) over epochs (left) and time (right) for the Fashion MNIST dataset.}
    \label{fig:fashion-records}
\end{figure}

Fig.~\ref{fig:fashion-records} shows the results of the different Riemannian optimizers and the conventional ADAM optimizer over 30 training epochs. Note that ADAM does not take the quotient structure into account and simply works with the unconstrained tensors of $\mathcal{E}$. After 30 epochs, all optimizers achieve approximately the same test accuracy of $\sim 89\%$. The lowest training loss was achieved by RMUONDOG, followed closely by RADAM, ADAM \& RMUON. RDOG shows the slowest convergence behavior, but the final test accuracy is still comparable. In terms of runtime, ADAM is clearly in the lead, although the Riemannian optimizers perform competitively. 

At first glance, it seems as if optimizing on $\cT$ and taking the quotient structure into account does not bring a significant advantage, as ADAM outperforms most of the Riemannian optimizers without the additional overhead of projections and retractions. However, for many downstream tasks such as model compression or the computation of explainability measures the TTN must be orthogonal \cite{scharf2026quantuminspiredrobustscalablesar}. Thus, the TTN optimized by ADAM must first be orthogonalized, which accumulates the norm of all non-isometric parts of the tensors in the root tensor. 
Indeed Riemannian optimizers keep the \emph{total norm}\footnote{For orthogonal TTN $\theta\in\cT$, the total norm $||\tau(\theta)||_F$ is equivalent to the norm of the root tensor $||B^{(N)}||_F$ as the isometries simply cancel out. For general TTN $\theta\in\mathcal{E}$ one can compute this quantity by orthogonalizing the TTN first.} $||\tau(\theta)||_F$ of the TTN model $\theta$ stable during training, while the norm explodes when optimizing with ADAM, exceeding the numerical range of \texttt{float32} after just six epochs (see Fig.~\ref{fig:norm} in appendix).
For inference this is not a problem since one can simply work with the numerically stable version $\theta\in\mathcal{E}$ that was learned. However, orthogonalizing this TTN can lead to numerical instabilities, which incur problems in downstream tasks. Fig.~\ref{fig:compression} demonstrates this using the example of model compression \cite{scharf2026quantuminspiredrobustscalablesar}, which results in significantly worse compression accuracies for ADAM when compared to the models that were optimized in $\cT$. 
An explanation for this behavior is the fact that the compression algorithm is based on singular value decompositions of the orthogonality center, which for ADAM is numerically unstable and possibly ill-conditioned due to the exploding norm. Another example for a downstream task would be the computation of feature entropies (see e.g. \cite{scharf2026quantuminspiredrobustscalablesar, felser2020quantuminspired}). It is algorithmically closely related to model compression and applying it to final iterates obtained by ADAM can be expected to yield unreliable values.

\begin{figure*}[t!]
    \centering
    \begin{subfigure}[b]{0.49\linewidth}
        \centering
        \includegraphics[height=2.3in]{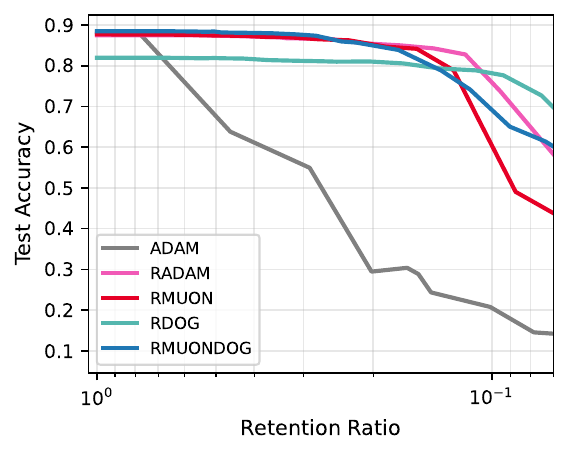}
        \caption{After three epochs.}
    \end{subfigure}
        ~ 
    \begin{subfigure}[b]{0.49\linewidth}
        \centering
        \includegraphics[height=2.3in]{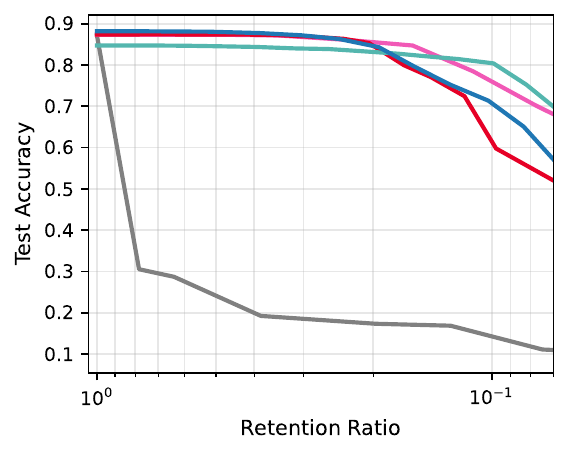}
        \caption{After four epochs.}
    \end{subfigure}
    \caption{Test accuracy for the Fashion MNIST dataset using compressed TTN models that were trained with different optimizers. The retention ratio of the compression is defined as the ratio of kept parameters to the original number of parameters. While the models that were trained with Riemannian optimizers in $\cT$ can be heavily compressed without a significant drop of the test accuracy, this is not the case for the model trained with ADAM. Furthermore, one can see that the problem gets worse with each epoch, as the total norm of the TTN continues to grow and the TTN moves further away from the orthogonal initial parametrization.}
    \label{fig:compression}
\end{figure*}

\subsection{CIFAR10 \& Imagenette}
For the CIFAR10 \cite{Krizhevsky2012} and Imagenette \cite{Howard2019} datasets, we use a hybrid neural network - tensor network architecture, where the neural network acts as a feature extractor and the tensor network as the classification head. The main reason for this is that a full tensor network is not able to generalize well as it lacks the translational invariance that is necessary for this dataset. To preserve the explainability of the tensor network classifier, we need to make sure that the extracted features are meaningful and interpretable, which is not the case for standard neural network backbones. 
Thus, we divide the input image into $p \times p$ \emph{patches}, that are processed separately by the same convolutional neural network (CNN). This produces an embedding vector of dimension $d = 64$ for each patch of the image, that will serve as the input for the tensor network classifier (see Fig.~\ref{fig:cnn-ttn}). 
The important distinction to conventional CNN feature extractors is that no information is exchanged between the individual patches, such that we can clearly trace back each feature vector to a well-defined region of the image. The exchange of information between the patches happens entirely in the tensor network, where one can analyze the learned patterns in a rigorous and precise manner (see e.g. \cite{scharf2026quantuminspiredrobustscalablesar, felser2020quantuminspired}). 
The exact architecture used for both datasets is summarized in the following table.
\begin{center}
\begin{tabular}{c|c|c|c|c|c|c|c}
    dataset & input res. & patches $p$ & conv. layers & sites $n$ & $d$ & $\chi_\mathrm{max}$ & tree-depth\\ \hline
    Fashion MNIST & $16 \times 16$ & N/A & None & 256 & 2 & 16 & 8 \\
    CIFAR10 & $32 \times 32$ & $4 \times 4$ & 4 & 16 & 64 & 12 & 4 \\
    Imagenette &  $160 \times 160$ & $8\times 8$ & 8 & 64 & 64 & 32 & 6\\
\end{tabular}
\end{center}
In all experiments, the CNN feature extractor is optimized simultaneously with the TTN. For the CNN itself, we employ the ADAM optimizer with standard parameters ($\alpha=0.001$, $\beta_1=0.9$ and $\beta_2=0.999$). In order to enrich the empirical data distribution, we augmented training images through random cropping, random erasing, and random horizontal flips.

Inspecting Figs.~\ref{fig:cifar10} \& \ref{fig:Imagenette} yields similar conclusions as for the Fashion MNIST dataset: Riemannian optimizers perform on par with ADAM in terms of test accuracy, with every optimizer reaching just above $\sim 80\%$ for CIFAR10 and just below $\sim 80\%$ for Imagenette, but ADAM is faster as it does not require projections or retractions. The adverse effect of ADAM on downstream compression is considerably less severe in the case of CIFAR10, as seen in Fig.~\ref{fig:compression_cifar}. We attribute this to the fact, that we used a smaller tensor network architecture for CIFAR10 than for Imagenette and Fashion-MNIST. In fact, the tree for the latter is twice as deep as for the former, which possibly amplifies numerical problems for algorithms that heavily rely on hierarchical structure of the network. This argument is supported by compression accuracies recovered from the model trained for Imagenette. Its larger depth renders the final iterate from ADAM incompressible.

\begin{figure}
    \centering
    \includegraphics[width=\linewidth]{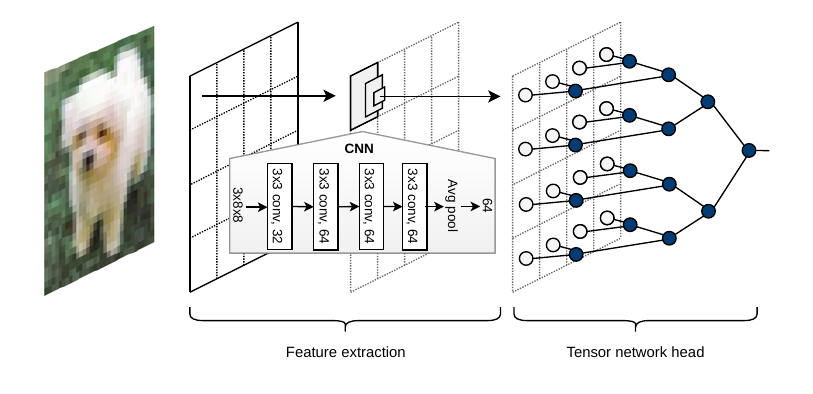}
    \caption{CNN-TTN hybrid architecture as employed for CIFAR10 and Imagenette datasets. The CNN is applied to each patch separately and acts as a feature map for the TTN head.}
    \label{fig:cnn-ttn}
\end{figure}

\begin{figure}
    \centering
    \begin{minipage}{0.5225\textwidth}
        \centering
        \includegraphics[width=\textwidth]{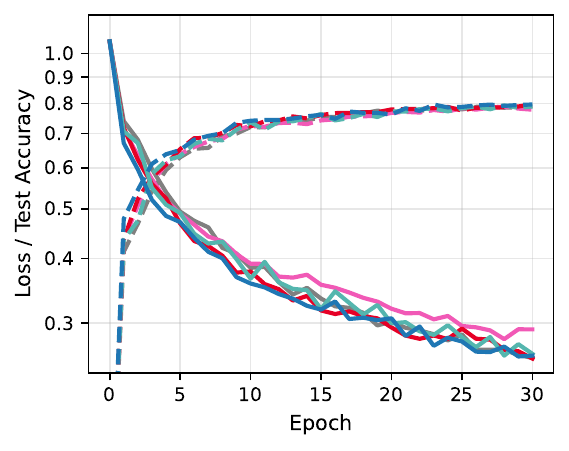}
    \end{minipage}\hfill
    \begin{minipage}{0.4775\textwidth}
        \centering
        \includegraphics[width=\textwidth]{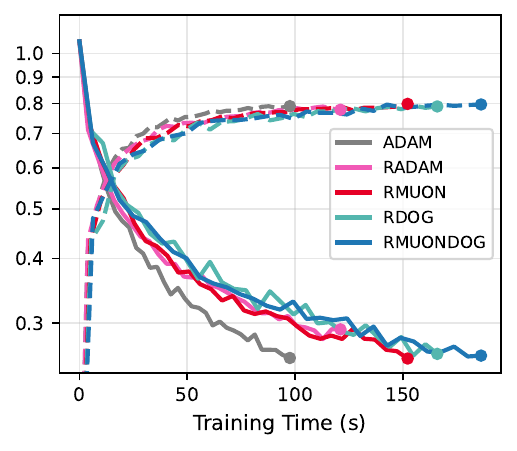}
    \end{minipage}
    \caption{Comparison of optimizers through loss (solid) and accuracy (dashed) over epochs (left) and time (right) on the CIFAR10 dataset.}
    \label{fig:cifar10}
\end{figure}

\begin{figure}
    \centering
    \begin{minipage}{0.5225\textwidth}
        \centering
        \includegraphics[width=\textwidth]{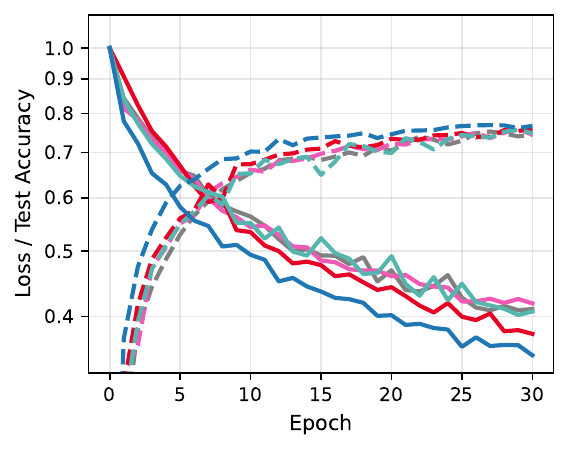}
    \end{minipage}\hfill
    \begin{minipage}{0.4775\textwidth}
        \centering
        \includegraphics[width=\textwidth]{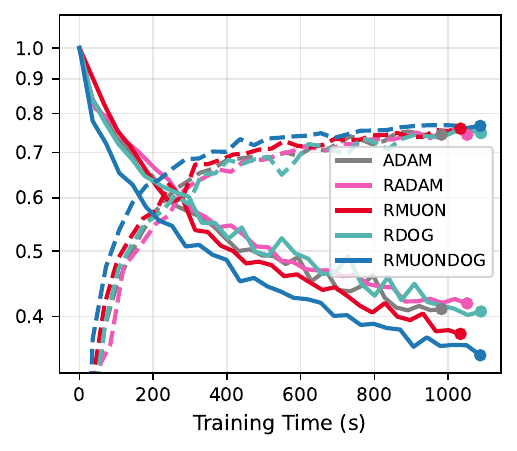}
    \end{minipage}
    \caption{Comparison of optimizers through loss (solid) and accuracy (dashed) over epochs (left) and time (right) on the Imagenette dataset.}
    \label{fig:Imagenette}
\end{figure}

\begin{figure}
    \centering
    \begin{minipage}{0.5\textwidth}
        \centering
        \includegraphics[width=\textwidth]{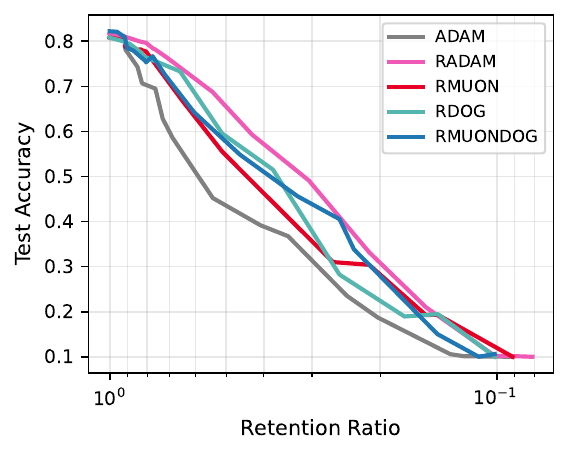}
    \end{minipage}\hfill
    \begin{minipage}{0.5\textwidth}
        \centering
        \includegraphics[width=\textwidth]{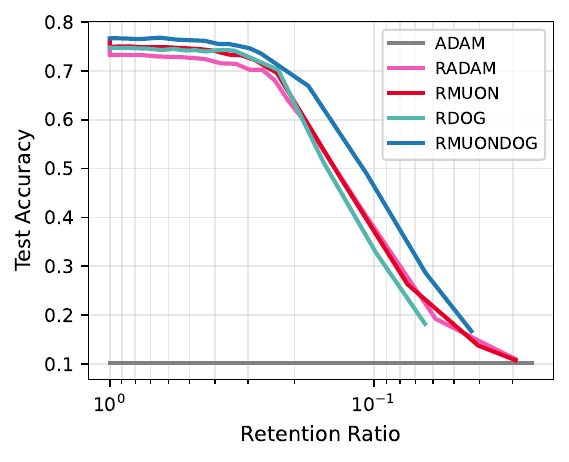}
    \end{minipage}
    \caption{Test accuracy for the CIFAR10 (left) and Imagenette (right) using compressed models that were trained for 30 epochs using different optimizers. Compressing the iterate obtained from ADAM is not possible: The model is rendered useless through orthogonalization alone.}
    \label{fig:compression_cifar}
\end{figure}

\section{Conclusion}\label{sec:conclusions}
In this work, we developed the theoretical framework for stochastic optimization on the manifold of orthogonal TTN for machine learning applications. Furthermore, we introduced a novel hybrid CNN-TTN framework that preserves the key advantages of tensor networks such as explainability and compressibility, while enabling the efficient processing of large-scale images. We compared different stochastic Riemannian optimizers against the conventional unconstrained Adam optimizer on three different image classification tasks.
While the Riemannian optimizers are comparable to Adam in terms of convergence and generalization, the Riemannian optimizers learn a more numerically stable version for downstream tasks. 

The main problem with unconstrained optimization, especially for deep TTN, is that the final iterates can no longer be orthogonalized correctly. Most tensor network algorithms, including the computation of explainability measures and model compression, however, are based on orthogonalization. Since these methods are the key advantages of tensor networks over neural networks, the computational and algorithmic overhead of Riemannian optimization is usually justified. Furthermore, by exploiting a fast approximate POGO retraction for the Stiefel manifold that profits from GPU acceleration, we were able to significantly reduce the computational overhead. In general, there is still room for improvement in the computational efficiency of the Riemannian optimizers, and it remains to be seen whether the gap to unconstrained optimization can be further narrowed so that the additional cost becomes negligible. Further analysis is also necessary in order to understand the origins of the numerical problems that arise from unconstrained optimization, and whether they can be remedied. If, for example, the instabilities of the orthogonalization can be traced back exclusively to the exploding norm of the root tensor, weight decay or other regularization methods could greatly improve stability. Another possibility would be to develop alternative algorithms for downstream tasks that do not rely on orthogonalization.

Another interesting research direction would be to develop stochastic optimization schemes for TN with adaptive bond dimensions, similar to two-site DMRG methods that are widely used in quantum simulation. Not only would this essentially combine optimization and compression into one process, but it also further reduces the number of hyperparameters that need to be tuned for the TTN and is shown to help escape local minima in ground state search\cite{Catarina2023}. It is likely that this would aid generalization in the machine learning setting, as the model is incentivized to only use the truly relevant information.
The main theoretical obstacle is that adapting the bond dimensions during optimization changes the manifold structure, so conventional Riemannian optimization methods do not apply anymore.

Finally, combining the expressive power of neural networks with the interpretability and flexibility of tensor networks has great potential, well beyond image classification. Tensor-network methods have already been explored in settings ranging from high-energy data analysis \cite{felser2020quantuminspired, Borella_2025} and SAR classification \cite{scharf2026quantuminspiredrobustscalablesar} to optimization problems in radiotherapy \cite{Cavinato2021}.
Expanding this approach to other applications and architectures could be a big step towards powerful yet transparent AI technologies.

\section*{Acknowledgements}
The authors would like to thank Tensor AI Solutions GmbH for supporting this work and providing the code
framework that allowed the numerical evaluation of our findings.

\paragraph{Author contributions}
M.S. and M.W. contributed in equal parts to this work. M.W. is responsible for the theoretical developments; M.S. devised the hybrid CNN-TTN architecture and conducted the numerical experiments. A.U. supervised the mathematical formulation and the theoretical foundations. T.F. and M.T. coordinated the overall work and provided the technical background and software on TN-machine learning. All authors contributed to the conceptualization and revision of the work.

\paragraph{Reproducibility statement}
To support reproducibility and verification, the full code used to produce our numerical data will be released upon journal acceptance.

\paragraph{Declaration of AI-assisted tools}
The plotting scripts to visualize numerical data were generated using ChatGPT. The authors remain fully responsible for the content, analyses, interpretations, and conclusions presented in this work.

\paragraph{Funding information}
The work of M.S., M.T., M.W. and T.F. was supported by Tensor AI Solutions GmbH. The work of A.U. was supported
by the Deutsche Forschungsgemeinschaft (DFG, German Research Foundation) – Projektnummer
506561557.

\begin{appendix}
\numberwithin{equation}{section}

\section{Geodesics of the TTN quotient} \label{sec:geodesic_distance_quotient}
The TTN quotient is closely related to the Cartesian product $\cZ$ of Grassmannians defined in~\eqref{eq:grassmann_product}.
Indeed, $\cZ$ is a group quotient to $\cT$, with the same Lie group~\eqref{eq: group G} as employed for $\cTG$, but to a different Lie group action: one that acts tensorwise instead of coupling all tensors. We define the quotient map
\begin{equation*}
    \pi^\cZ: \cT \to \cZ, \qquad (B^{(t)})_{t \in T} \mapsto \begin{cases}
        \pi^\Gr(B^{(t)}) & \text{for } t \in T^*, \\
        B^{{(t)}} &\text{else},
    \end{cases}
\end{equation*}
where
$\pi^\Gr: \St(n,k) \to \Gr(n,k)$ is the Grassmann quotient map.
Importantly, this means that the vertical spaces corresponding to $\cZ$ and $\cTG$ differ, but their horizontal distributions can be defined to coincide by taking the Cartesian horizontal space~\eqref{eq:horizontal_space} in both cases. Based on this, we can derive an important relation between $\cTG$ and $\cZ$, as seen in the following theorem.
\begin{theorem} \label{thm:geodesics}
    Any geodesic $\hat \gamma: [0, 1] \to \cTG$ can be mapped to a geodesic $\tilde \gamma: [0, 1] \to \cZ$ of the same length, in such a way that both geodesics share the same representation in $\cT$, meaning there exists a horizontal curve $\gamma: [0, 1] \to \cT$ that satisfies $\pi(\gamma(s)) = \hat{\gamma}(s)$ and $\pi^\cZ(\gamma(s)) = \tilde \gamma(s)$ for all $s \in [0, 1]$.
\end{theorem}

\begin{proof}
Let $\hat{\gamma}(0) = [\theta]$ and $\hat{\gamma}(1) = [\theta']$. 
Since geodesics are characterized as curves that locally minimize length, we assume that $L(\hat \gamma) = d^\cTG([\theta], [\theta'])$. If this assumption is not met,  subdivide $\hat \gamma$ into sufficiently short segments that each fulfill it, and apply the upcoming argument to all segments.
According to~\cite[Prop. II.3.1]{Kobayashi1963} for any representative $\theta \in \cT$ of $[\theta]$, there exists a unique horizontal lift $\gamma$ of $\hat{\gamma}$ which satisfies
$\hat \gamma(s) = \pi(\gamma(s))$, $\gamma(0) = \theta$ and
$\gamma'(s) = \lift_{\gamma(s)}(\hat \gamma'(s)) \in H_{\gamma(s)}\cT$. Since 
horizontal lifts of curves preserve lengths~\cite[Lem. 26.11.1]{Michor2008}, the curve $\gamma$ will be the solution of $\min_{\gamma \text{ horizontal}} L(\gamma)$.
By the form of the horizontal space~\eqref{eq:horizontal_space}, for $\gamma = (\gamma_{(t)})_{t \in T}$, each individual component except for the root is horizontal, meaning $\gamma'_{(t)}(s) \in H_{\gamma_{(t)}(s)}\St(\chi_{t_L}\chi_{t_R},\chi_t)$. As $\cT$ forms as a Cartesian product (meaning that in particular the Riemannian metric decouples) and since there are no coupling constraints between components of $\gamma$, each non-root component will be the solution of $\min_{\gamma_{(t)}\text{ horizontal}}L(\gamma_{(t)})$ with the corresponding endpoints and the root component of $\gamma$ will solve
$\min_{\gamma_{(N)}} L(\gamma_{(N)})$.
Thus, the $\gamma_{(t)}$ of each non-root node is a horizontal geodesic in $\St(\chi_{t_L}\chi_{t_R},\chi_t)$. Since $H_{\gamma_{(t)}(s)}\St$ is orthogonal to the vertical space associated with the Grassmannian, $\pi^\Gr(\gamma_{(t)})$ is a geodesic in $\Gr(\chi_{t_L}\chi_{t_R},\chi_t)$~\cite[Lemma 26.11.4]{Michor2008}. Applying this logic to all non-root components, shows that $\tilde \gamma(s)$ is a geodesic in $\cZ$, and its length evaluates as $L(\tilde \gamma) = L(\gamma)$. 
\end{proof}

At this point, we remind of the nuance concerning the full multilinear rank constraints of $\cTG$ (see Section \ref{sec:quotientspace}): Not all geodesics in $\cZ$ lift to geodesics in $\cT$ that satisfy it, meaning they have no correspondence in $\cTG$. 
Thus the above argument cannot be reversed. Still, the equivalence of geodesics can also be constructed with the following high-level argument. If $\nabla$ is the Levi-Civita connection on $\cT$, then the Levi-Civita connection $\nabla^\cZ$ on $\cZ$ is given by $$\lift^\cZ(\nabla^\cZ_U V) = \Proj^H(\nabla_{\lift^\cZ(U)} \lift^\cZ(V))$$ for $U, V$ vector fields on $\cZ$. This follows since $H_\theta\cT$ is orthogonal to the vertical spaces corresponding to $\cZ$~\cite[Thm. 9.34]{Boumal2023}. The same formula also yields a connection $\nabla^\cTG$ on $\cTG$ given by $$\lift(\nabla^\cTG_U V) = \Proj^H(\nabla_{\lift(U)} \lift(V))$$ for $U, V$ vector fields on $\cTG$~\cite[Thm. 3]{Willner2025}. Even though this connection has torsion, because geodesics only depend on the symmetric part of a connection, the geodesic spray of this connection will coincide with the Levi-Civita one of $\cTG$. Thus it is clear that $\cTG$ and $\cZ$ share the same geodesics, given through the same horizontal parametrization in $\cT$.

This argument also implies that geodesic distances and exponential maps of $\cTG$ and $\cZ$ share representations in $\cT$.
Indeed, an important corollary of the above theorem is equation~\eqref{eq:quotient_geodesic_distance}. It can be found by employing that 
\begin{equation*}
    L(\hat \gamma)^2 = L(\gamma)^2 = \sum_{t \in T^*} L(\gamma_{(t)})^2 + L(\gamma_{(N)})^2,
\end{equation*}
and plugging in the respective expressions for geodesic distances. Furthermore, we can conclude that the exponential map on $\cTG$ is related to the Grassmann exponential map, which in term of Stiefel representatives reads
\begin{equation*}
        \mathrm{Exp}^\Gr_X(\delta X) = X V \cos(\Sigma) V^T + U \sin(\Sigma) V^T,
    \end{equation*}
where $\delta X = U \Sigma V^T$ is the compact SVD of $\delta X \in T_X\St(n,k)$.

\begin{corollary}
Let $\theta = (B^{(t)})_{t \in T}$ be a representative of $[\theta] \in \cTG$ and let the horizontal lift of $\hat \xi \in T_{[\theta]}\cTG$ be given by $\lift_\theta(\hat \xi) = (\delta B^{(t)})_{t \in T}$.
    The exponential map on $\cTG$ reads
    \begin{equation} \label{eq:exp}
        \mathrm{Exp}_{[\theta]} = \pi \circ \mathrm{R}^\mathrm{Exp}_\theta \circ \lift_\theta,
    \end{equation}
    where 
    \begin{equation*}
        \mathrm{R}^\mathrm{Exp}_\theta(\lift_\theta(\hat \xi)) = 
        \begin{cases}
            \mathrm{Exp}^\Gr_{B^{(t)}}(\delta B^{(t)}) & \text{if } t \in T^*,\\
            B^{(t)} + \delta B^{(t)} & \text{else}.
        \end{cases}
    \end{equation*}
\end{corollary}
\begin{proof}
    Theorem~\ref{thm:geodesics} shows that when fixing a base point $\theta$, any geodesic $\hat \gamma: [0, 1] \to \cTG$ can be mapped to a geodesic $\tilde \gamma: [0,1] \to \cZ$, 
    such that both geodesics share the same representative $\gamma: [0, 1] \to \cT$ at all points, meaning $\pi(\gamma(s)) = \hat \gamma(s)$ and $\pi^\Gr(\gamma_{(t)}(s)) = \tilde \gamma_{(t)}(s)$ for all $s \in [0, 1]$. 
    Furthermore if $\hat \gamma'(0) = \hat \xi$, then $\gamma'(0) = \lift_\theta(\hat \xi) = (\delta B^{(t)})_{t \in T}$ and $\tilde \gamma_{(t)}'(0) = \DD \pi^\Gr(B^{(t)})[\delta B^{(t)}]$ for all $t \in T^*$. 
    Because $\tilde \gamma$ is a geodesic, it holds that $\tilde \gamma_{(t)}(1) = \mathrm{Exp}_{B^{(t)}}(\DD \pi^\Gr(B^{(t)})[\delta B^{(t)}])$. By lifting all components, we see that $\gamma(1) = \mathrm{R}^\mathrm{Exp}_\theta((\delta B^{(t)})_{t \in T}) = \mathrm{R}^\mathrm{Exp}_\theta(\lift_\theta(\hat \xi))$. 
    Finally, going back to the TTN quotient, it holds that $\hat \gamma (1) = \pi(\mathrm{R}^\mathrm{Exp}_\theta(\lift_\theta(\hat \xi)))$. Thus by definition $\pi \circ \mathrm{R}^\mathrm{Exp}_\theta \circ \lift_\theta$ is the exponential map on $\cTG$.
\end{proof}

\section{Convergence of RADAGRAD} \label{sec:radagrad_conv}
Let $\hat f: \cTG \rightarrow \R, \, \hat f = \mathcal L\circ \hat \tau$ that is accessible through function oracles $\hat f_k$ by evaluating the loss only on a random batch of training data. In this section we present a proof of convergence for RADAGRAD on the TTN quotient manifold, by bounding the regret 
\[
R_K = \sum_{k=1}^K \hat f_k ([\theta_k]) - \hat f_k([\theta_*]),
\]
where $[\theta_*]=  \mathrm{argmin}_{[\theta] \in \cTG} \sum_{k=1}^K \hat f_k([\theta])$ is an oracle minimizer . The proof can be extended to cover more advanced algorithms like RADAMNC or RAMSGRAD by going along the lines of~\cite{Becigneul2018}. Here we focus on the less involved RADAGRAD to more clearly state the peculiarities concerning quotient optimization.
\begin{algorithm}
\caption{RADAGRAD for TTNs} \label{alg:radagrad}
\begin{algorithmic}
    \Require $\theta_1 = (B^{(t)}_1)_{t \in T} \in \cT$, learning rate $\alpha$, projector P\,, retraction R
    \State Set $v_0 = (0)_{t \in T}$
    \For{$k = 1, 2, \ldots$}
    \State $(G_k^{(t)})_{t \in T} = \mathrm P (\nabla  f_k(\theta_k))$
    \For{$t \in T$}
        \State $v_k^{(t)} = v_{k-1}^{(t)} + \langle G_k^{(t)} \mid G_k^{(t)}\rangle$
        \State $\delta B^{(t)}_{k} = -\alpha\, G_k^{(t)}/\sqrt{v_k^{(t)}}$
    \EndFor
    \State $\theta_{k+1} = \mathrm R_{\theta_k}((\delta B^{(t)}_{k})_{t \in T}) $
    \EndFor
\end{algorithmic}
\end{algorithm}
\noindent We make the following assumptions:
\begin{enumerate}[label=(A\arabic*)]
    \item The step-size $\alpha$ of Algorithm~\ref{alg:radagrad} satisfies $\alpha \|G_1^{(t)}\|/\sqrt{v_1^{(t)}} \leq \frac{\pi}{2}$ for all $t \in T^*$ and $G_k^{(t)}$ remains bounded for $t \in T$.
    \item Algorithm~\ref{alg:radagrad} runs employing the projector $\mathrm{P} = \Proj^H$ and the retraction $\mathrm{R} = \mathrm{R}^\mathrm{Exp}$ that descends to the exponential map on $\cTG$.
    \item $[\theta_*]$ and all iterates $[\theta_k]$ of Algorithm~\ref{alg:radagrad} reside in a compact and geodesically convex subset of $\cTG$ with diameter bounded by $D_\infty$, and $\hat f$ is geodesically convex on that subset.
\end{enumerate}
A few comments about the validity of the assumptions are in order. Assumption (A1) essentially requires that all gradients remain bounded, which is a standard assumption for regret-based convergence proofs. In this case, a step size $\alpha > 0$ can be found that satisfies the first condition.  Assumption (A2) considers the exponential map as a retraction. Since other retractions on $\cTG$ actually match with the exponential map up to second order, Assumption (A2) is still satisfied in approximation for those. Finally, the existence of geodesically convex subsets on $\cTG$ on which $\hat f$ is geodesically convex as required in 
(A3) can be guaranteed at least locally around points in which the Riemannian Hessian $\mathrm{Hess}\, \hat f$ is positive definite (see \cite{Boumal2023} for a definition of the Riemannian Hessian). In the case of strongly convex loss functions $\mathcal L$ like the $\mathcal L_2$ loss, an interesting case arises if the oracle minimizer $[\theta_*]$ provides a critical point of~$\mathcal L$. The following result is in analogy to~\cite[Theorem~2.11]{Rohwedder2013}. 

\begin{proposition}\label{prop:Riemannian Hessian}
    Let $\hat f = \mathcal{L} \circ \hat \tau : \cTG \to \R$ and let $\mathcal{L}:\R^{d_1\times\cdots\times d_n\times d_{\text{out}}}\to \R$ be a twice differentiable strongly convex function. Then $\mathrm{Hess}\, \hat f([\theta_*])$ is positive definite at any point $[\theta_*] \in \cTG$ satisfying $\nabla\mathcal{L}(\hat \tau([\theta_*])) = 0$.
\end{proposition}

\begin{proof}
Define the space $\mathcal{H} \coloneq \hat \tau(\cTG) \subset \R^{d_1\times\cdots\times d_n\times d_{\text{out}}}$ as the manifold of tensors with fixed hierarchical rank $(\chi_t)_{t \in T^*}$ \cite{Uschmajew:2013}.
    Let $\hat \gamma: [0, 1] \to \cTG$ be a geodesic on the quotient with $\hat{\gamma}(0) = [\theta_*]$, and define $\Gamma = \hat \tau \circ \hat \gamma$ the respective curve in low-rank tensor manifold $\mathcal{H}$, so we have the equality $\hat f \circ \hat \gamma  = \mathcal L \circ \Gamma$. Then according to Taylor's theorem on Riemannian manifolds~\cite[Eq.~5.26]{Boumal2023}, it holds that 
    \begin{equation*}
        (\hat f \circ \hat \gamma)''(0) = \hat \rho_{\hat \gamma(0)} (  \mathrm{Hess}\, \hat f(\hat \gamma(0))[\hat \gamma'(0)] , \hat \gamma'(0)) +  \hat \rho_{\hat \gamma(0)}( \grad \hat f(\hat \gamma(0)) , \hat \gamma''(0)),
    \end{equation*}
    where $\hat \rho$ denotes the quotient metric \eqref{eq:quotient_metric}. Because $\hat \gamma$ is a geodesic, its velocity is constant, and the second term vanishes. We apply Taylor's theorem again to find
    \begin{equation*}
        (\mathcal L \circ \Gamma)''(0) = \langle  \Gamma'(0) | \mathrm{Hess}\, \mathcal{L}(\Gamma(0)) | \Gamma'(0)\rangle + \langle  \grad \mathcal{L }(\Gamma(0)) | \Gamma''(0)\rangle.
    \end{equation*}
    Here, in slight abuse of notation $\mathrm{Hess}\, \mathcal{L}$ (resp. $ \grad \mathcal{L }$) denotes the Riemannian Hessian (resp. gradient) of $\mathcal{L}$ restricted to $\mathcal{H}$. Recall that $\nabla\mathcal{L}(\hat \tau([\theta_*])) = 0$. Since $\mathcal{H}$ is an embedded submanifold of $\R^{d_1\times\cdots\times d_n\times d_{\text{out}}}$ equipped with the Euclidean metric \cite{Uschmajew:2013}, we can conclude from~\cite[Prop. 3.61]{Boumal2023} that $\grad \mathcal{L }(\Gamma(0)) = 0$ as well and from~\cite[Eq.~ 5.34]{Boumal2023} that $\mathrm{Hess}\, \mathcal{L}$ coincides with the Euclidean Hessian $\mathrm{H}_\mathcal{L}$ on the respective tangent space of $\mathcal H$ at $\Gamma(0)$. By assumption, $\mathrm{H}_\mathcal{L}$ is positive definite, therefore $\mathrm{Hess}\, \hat f(\hat \gamma(0))$ is positive definite as well.
\end{proof}

\begin{theorem}
Let the iterates $\theta_1, \theta_2, \ldots, \theta_K \in \cT$ generated by Algorithm~\ref{alg:radagrad} be representatives of $[\theta_1]$, $[\theta_2], \ldots, [\theta_K] \in \cTG$.
Under assumptions (A1-A3), the regret on the quotient is bounded by  
\begin{equation*}
    R_K = \sum_{k=1}^K \hat f_k ([\theta_k]) - \hat f_k ([\theta_*]) \leq 
    \left(\frac{D_\infty^2}{2\alpha} + \alpha\right) \sum_{t \in T} \sqrt{\sum_{k=1}^K \|G^{(t)}_{k}\|^2 } \; .
\end{equation*}
For $K \to \infty$, this implies convergence of objective values in expectation.
\end{theorem}

\begin{proof}
Consider the iterates $[\theta_k], [\theta_{k+1}]$ and the oracle minimizer $[\theta_*]$. Their representatives are denoted by $\theta_k = (B^{(t)}_{k})_{t \in T}$, $\theta_{k+1} = (B^{(t)}_{k+1})_{t \in T}$ and 
$\theta_* = (B^{(t)}_{*})_{t \in T}$. It holds that $\theta_{k+1} = \mathrm R_{\theta_k}((-\alpha\, G_k^{(t)}/(v_k^{(t)})^{1/2})_{t \in T})$, where $(G_k^{(t)})_{t \in T} = \lift_{\theta_k}(\grad \hat f([\theta_k]))$. Let $\hat \gamma: [0, 1] \to \cTG$ be a geodesic connecting $[\theta_k]$ and $[\theta_*]$. We denote $\gamma$ the horizontal lift of $\hat{\gamma}$, and we set $(\Delta_k^{(t)})_{t \in T} = \gamma'(0) = \lift_{\theta_k}(\hat \gamma'(0)) \in H_{\theta_k}\cT$. By employing the geodesic convexity of $\hat f$ and lifting the metric we find that
\begin{align} \label{eq:regret}
    \sum_{k=1}^K \hat f_k ([\theta_k]) - \hat f_k ([\theta_*]) \leq \sum_{k=1}^K \hat \rho_{[\theta_k]}( - \grad \hat f([\theta_k]), \hat \gamma'(0)) 
    = \sum_{k=1}^K \sum_{t \in T} \langle - G_k^{(t)} | \Delta_k^{(t)} \rangle.
\end{align}
We now let $[B^{(t)}_{k}] = \pi^\Gr(B^{(t)}_{k})$, $[B^{(t)}_{k+1}] = \pi^\Gr(B^{(t)}_{k+1})$ and $[B^{(t)}_{*}] = \pi^\Gr(B^{(t)}_{*})$. As seen in Theorem~\ref{thm:geodesics}, any non-root component of $\gamma$ descends to a geodesic $\pi^\Gr(\gamma_{(t)})$ on the Grassmannian $\Gr(\chi_{t_L}\chi_{t_R},\chi_t)$, therefore
\begin{align*}
    [B^{(t)}_{*}] &= \mathrm{Exp}_{[B^{(t)}_{k}]}(\DD \pi^\Gr(B^{(t)}_{k})[\Delta^{(t)}_{k} ]), \\
    [B^{(t)}_{k+1}] &= \mathrm{Exp}_{[B^{(t)}_{k}]}(-\alpha/(v_k^{(t)})^{1/2}\, \DD \pi^\Gr(B^{(t)}_{k})[G^{(t)}_{k} ]).
\end{align*}
The second equality follows from the special choice of our retraction (A2). The geometric interpretation of the metric then yields
\begin{align*}
    \langle -\alpha/(v_k^{(t)})^{1/2} G^{(t)}_{k} | \Delta_k^{(t)}\rangle 
    &=
    \rho^\Gr_{[B^{(t)}_{k}]}(-\alpha/(v_k^{(t)})^{1/2}\, \DD \pi^\Gr(B^{(t)}_{k})[G^{(t)}_{k} ], \DD \pi^\Gr(B^{(t)}_{k})[\Delta^{(t)}_{k} ]) \\
    &= d^\Gr([B^{(t)}_{k}], [B^{(t)}_{k+1}]) d^\Gr([B^{(t)}_{k}], [B^{(t)}_{*}])\cos(\angle G^{(t)}_{k} \Delta_k^{(t)}).
\end{align*}
The Grassmann manifold is of non-negative sectional curvature~\cite[Prop.~4.1]{Bendokat2024}. Thus, the cosine law known from Euclidean spaces directly extends, and we can write 
\begin{align*}
    2\langle -\alpha/(v_k^{(t)})^{1/2} G^{(t)}_{k} | \Delta_k^{(t)}\rangle 
    &\leq d^\Gr([B^{(t)}_{k}], [B^{(t)}_{*}])^2 - d^\Gr([B^{(t)}_{k+1}], [B^{(t)}_{*}])^2 + d^\Gr([B^{(t)}_{k}], [B^{(t)}_{k+1}])^2 \\
    &= d^\Gr([B^{(t)}_{k}], [B^{(t)}_{*}])^2 - d^\Gr([B^{(t)}_{k+1}], [B^{(t)}_{*}])^2 + \alpha^2/v_k^{(t)} \langle  G^{(t)}_{k} | G^{(t)}_{k}\rangle.
\end{align*}
In the last equality, we used Assumption (A1), and that the injectivity radius of the Grassmann manifold is $\frac\pi2$. (A1) applies to all iterates instead of only the first one, as $\|G_k^{(t)}\|/\sqrt{v_k^{(t)}}$ is a decreasing sequence. At the root node, which resides in a linear space, a similar consideration holds. Using the notation from~\eqref{eq:node_distance}, we can subsume
\begin{equation}\label{eq: preestimate}
\begin{aligned}
    \sum_{k=1}^K \sum_{t \in T} \langle - G_k^{(t)} | \Delta_k^{(t)} \rangle 
    \leq \sum_{k=1}^K \sum_{t \in T} &
    \frac{(v_k^{(t)})^{1/2}}{2 \alpha} (d^{(t)}(B^{(t)}_{k}, B^{(t)}_{*})^2 - d^{(t)}(B^{(t)}_{k+1}, B^{(t)}_{*})^2) \\
    &+ \frac{\alpha}{2 (v_k^{(t)})^{1/2}} \|G^{(t)}_{k}\|^2 .
\end{aligned}
\end{equation}
Recall that $d^{(t)}(B^{(t)}_{k}, B^{(t)}_{*})^2 \leq \sum_{t \in T} d^{(t)}(B^{(t)}_{k}, B^{(t)}_{*})^2 = d^{\cTG}([\theta_k], [\theta_*])^2 \leq D_\infty^2$ by (A3), which we can use to further bound the first term on the right side of~\eqref{eq: preestimate}. Since $v_0^{(t)} = 0$ and $v_k^{(t)}$ is monotonically increasing, we can conclude that
\begin{equation} \label{eq:regret1}
\begin{split}
    & \sum_{t \in T} \sum_{k=1}^K 
    \frac{(v_k^{(t)})^{1/2}}{2 \alpha} (d^{(t)}(B^{(t)}_{k}, B^{(t)}_{*})^2 - d^{(t)}(B^{(t)}_{k+1}, B^{(t)}_{*})^2) \\
    & = \frac{1}{2\alpha} \sum_{t \in T} \sum_{k=1}^K ((v_k^{(t)})^{1/2} - (v_{k-1}^{(t)})^{1/2}) d^{(t)}(B^{(t)}_{k}, B^{(t)}_{*})^2 \\
    & \leq \frac{1}{2\alpha} \sum_{t \in T} \sum_{k=1}^K ((v_k^{(t)})^{1/2} - (v_{k-1}^{(t)})^{1/2}) D_\infty^2 \\
    &= \frac{D_\infty^2}{2\alpha} \sum_{t \in T} (v_K^{(t)})^{1/2} = \frac{D_\infty^2}{2\alpha} \sum_{t \in T}\sqrt{\sum_{k=1}^K \|G^{(t)}_{k}\|^2 } \; .
\end{split}
\end{equation}
For the second term in~\eqref{eq: preestimate}, Auer's lemma \cite{Auer2002} yields that
\begin{align} \label{eq:regret2}
    \sum_{t \in T}\sum_{k=1}^K\frac{\alpha}{2 (v_k^{(t)})^{1/2}} \|G^{(t)}_{k}\|^2 = 
    \frac{\alpha}{2 } \sum_{t \in T}\sum_{k=1}^K\frac{\|G^{(t)}_{k}\|^2}{\sqrt{\sum_{j=0}^k \|G^{(t)}_{j}\|^2 }} \leq 
    \alpha \sum_{t \in T} \sqrt{\sum_{k=1}^K \|G^{(t)}_{k}\|^2 } \; .
\end{align}
Plugging~\eqref{eq: preestimate},~\eqref{eq:regret1} \&~\eqref{eq:regret2} into~\eqref{eq:regret} yields the required result.
\end{proof}

\section{Additional numerical results}
\begin{figure}[H]
    \centering
    \begin{minipage}{0.5225\textwidth}
        \centering
        \includegraphics[width=\textwidth]{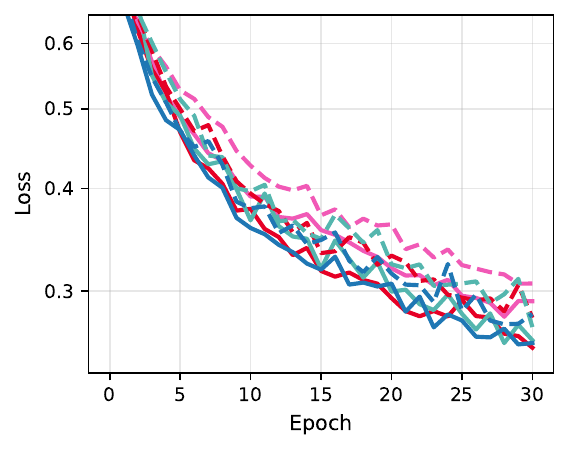}
    \end{minipage}\hfill
    \begin{minipage}{0.4775\textwidth}
        \centering
        \includegraphics[width=\textwidth]{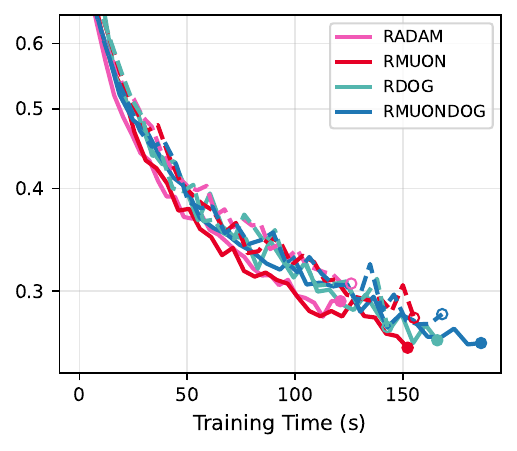}
    \end{minipage}
    \caption{Comparison of optimizers iterating on $\cT$ (dashed) versus optimizers respecting the quotient $\cTG$ (solid), exemplified on the CIFAR10 dataset. Algorithms on $\cT$ perform slightly worse in terms of final loss and test accuracy (not displayed). RADAM and RMUON are slightly faster on $\cTG$ because $\Proj^H$ is cheaper than $\Proj$, DOG-like algorithm are faster on $\cT$ because the evaluation of $q_k$ is cheaper than $d^\Gr$.}
    \label{fig:cifar10T}
\end{figure}

\begin{figure}[H]
    \centering
    \includegraphics[width=0.55\linewidth]{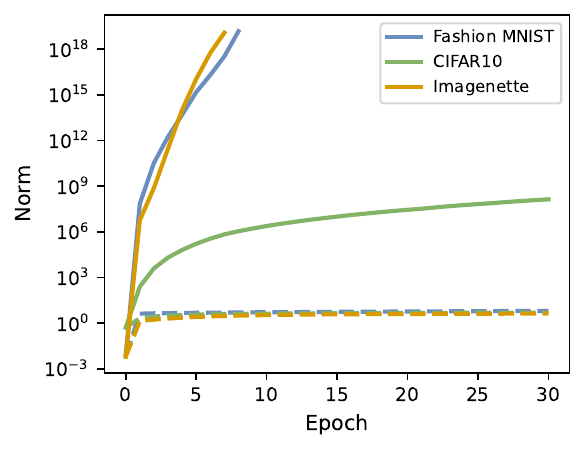}
    \caption{Development of the tensor norm $\|\tau(\theta_k)\|_f$ for ADAM (solid) and RADAM (dashed). Exploding norms hint at numerical problems for algorithms that act on $\tau(\theta)$, like compression. Riemannian optimizers (expemplified by RADAM) keep this norm stable.}
    \label{fig:norm}
\end{figure}

\end{appendix}

\printbibliography
\end{document}